\documentclass[a4paper]{article}

\usepackage[T1]{fontenc}
\usepackage[margin=25mm]{geometry}
\usepackage{amsmath,amssymb,amsthm}
\usepackage{stix}
\usepackage[authoryear,round,semicolon,longnamesfirst]{natbib}
\usepackage[hidelinks]{hyperref}
\usepackage{amsmath,amssymb,amsfonts,amsthm,mathtools}
\usepackage{bm}
\usepackage{booktabs}
\usepackage{enumitem}
\usepackage[dvipsnames]{xcolor}
\usepackage{aliascnt}

\usepackage[capitalise]{cleveref}

\numberwithin{equation}{section}

\theoremstyle{plain}
\newtheorem{theorem}{Theorem}[section]

\newaliascnt{proposition}{theorem}
\newtheorem{proposition}[proposition]{Proposition}
\aliascntresetthe{proposition}

\newaliascnt{lemma}{theorem}
\newtheorem{lemma}[lemma]{Lemma}
\aliascntresetthe{lemma}

\theoremstyle{definition}
\newtheorem{assumption}{Assumption}
\crefname{assumption}{Assumption}{Assumptions}
\Crefname{assumption}{Assumption}{Assumptions}

\newaliascnt{definition}{theorem}
\newtheorem{definition}[definition]{Definition}
\aliascntresetthe{definition}

\newtheorem{example}{Example}
\crefname{example}{Example}{Examples}
\Crefname{example}{Example}{Examples}

\theoremstyle{remark}
\newaliascnt{remark}{theorem}
\newtheorem{remark}[remark]{Remark}
\aliascntresetthe{remark}

\newcommand{\R}{\mathbb R}
\newcommand{\E}{\mathbb E}
\newcommand{\Prob}{\mathbb P}
\newcommand{\HS}{\mathcal L_2}
\newcommand{\ip}[2]{\left\langle #1,#2\right\rangle}
\newcommand{\norm}[1]{\left\|#1\right\|}
\newcommand{\Tr}{\operatorname{Tr}}
\newcommand{\trans}{\mathsf{T}}
\newcommand{\Ran}{\operatorname{Ran}}
\newcommand{\Law}{\operatorname{Law}}
\newcommand{\divg}{\operatorname{div}}
\newcommand{\FC}{\mathcal{FC}_b^\infty}
\newcommand{\KL}{D_{\mathrm{KL}}}
\newcommand{\KLdiv}[2]{\KL\left(#1\,\middle\|\,#2\right)}
\newcommand{\epr}{\mathrm{epr}}
\newcommand{\rev}{\mathrm{rev}}
\newcommand{\irr}{\mathrm{irr}}
\newcommand{\Frev}{F_{\mathrm{rev}}}
\newcommand{\Firr}{F_{\mathrm{irr}}}

\DeclareMathOperator{\Exp}{Exp}

\newcommand{\fin}{\mathrm{f.}}

\newcommand{\mH}{\mathcal H}
\newcommand{\mU}{\mathcal U}
\newcommand{\mJ}{\mathcal J}
\newcommand{\mN}{\mathbb N}
\newcommand{\mR}{\mathbb R}
\newcommand{\mP}{\mathbb P}
\newcommand{\mQ}{\mathbb Q}
\newcommand{\md}{\mathop{}\!\mathrm{d}} 
\newcommand{\me}{\mathrm{e}}

\allowdisplaybreaks

\hypersetup{
  pdftitle={Entropy Production for Stationary Diffusions on Hilbert Spaces},
  pdfauthor={Zhuoyuan Li, Yue Zhao, Aiqing Zhu},
  pdfkeywords={Entropy production, Time reversal, Hilbert-space diffusion}
}

\usepackage{authblk}

\title{Entropy Production for Stationary Diffusions on Hilbert Spaces}
\author[1]{Zhuoyuan Li%
\thanks{The authors are listed in alphabetical order.\\
  Email addresses: \href{mailto:zy.li@nus.edu.sg}{\texttt{zy.li@nus.edu.sg}} (Z.L.),
  \href{mailto:yuezhao@nus.edu.sg}{\texttt{yuezhao@nus.edu.sg}} (Y.Z.),
  \href{mailto:zaq@nus.edu.sg}{\texttt{zaq@nus.edu.sg}} (A.Z.).
  }%
}
\author[1]{Yue Zhao\protect\footnotemark[1]}
\author[2]{Aiqing Zhu\protect\footnotemark[1]}

\affil[1]{
Institute for Functional Intelligent Materials, National University of Singapore,
Singapore 117544, Singapore
}
\affil[2]{
Department of Mathematics, National University of Singapore,
Singapore 119076, Singapore
}

\date{}

\begin{document}
\maketitle
\vspace{-2em}
\begin{abstract}
We study entropy production for stationary diffusions on separable Hilbert spaces with possibly degenerate, state-dependent trace-class covariance.
We first define the reversible--irreversible drift decomposition to provide a candidate drift for the reversed dynamics.
Then by working directly with the invariant measure, we establish a lower bound in terms of the extended stationary Cameron--Martin energy of the irreversible drift, without requiring a diffusion representation of the stationary reversal.
The bound implies infinite entropy production when the energy is infinite and we establish sufficient conditions for equality in the finite-energy regime.
We also develop complementary criteria for identifying the stationary reversal as a Hilbert-space diffusion with constant or continuous state-dependent diffusion coefficients, respectively.
Two nonlinear infinite-rank examples are provided to show that finite entropy production does not require the irreversible drift to lie in the covariance range, whereas entropy production may be infinite even with globally Lipschitz coefficients and injective covariance.

\end{abstract}

\section{Introduction}\label{sec:introduction}
Stochastic differential equations provide a natural description of physical systems whose evolution is influenced by random fluctuations \citep{Seifert2012}.
In a steady regime, the probability distribution of the state remains constant in time, while individual configurations continue to evolve.
However, the statistics of these trajectories need not be preserved when their temporal order is reversed.
This asymmetry is a feature of the dynamics that the stationary distribution alone does not reveal \citep[Chapter~4]{JiangQianQian2004}.
Its quantitative description involves comparing the probability laws of forward and time-reversed trajectories.
Relative entropy measures the statistical distinguishability of these laws, and its value per unit time defines the entropy production rate (EPR).

For a stationary time-homogeneous Markov process with continuous paths, let $\mP_{[0,T]}$ denote its path law on $[0,T]$, and let $\mP^{\mathrm{rev}}_{[0,T]}$ be its image under $\omega\mapsto(\omega_{T-t})_{0\le t\le T}$ as the time-reversed path law.
The entropy production rate (EPR) is defined as
\begin{equation}\label{eq:intro-entropy-production-rate}
    \epr\coloneqq\lim_{T\to+\infty}\frac{1}{T}\KLdiv{\mP_{[0,T]}}{\mP^{\mathrm{rev}}_{[0,T]}}\in [0,+\infty],
\end{equation}
where $\KLdiv{\cdot}{\cdot}$ denotes relative entropy.
It can be shown that the normalized relative entropy above is independent of $T$ in this setting, so the limit is superfluous.
EPR provides a quantitative measure of the time asymmetry of stationary dynamics.
In particular, it distinguishes reversible from irreversible dynamics even when they share the same invariant measure \citep[Chapter~4]{JiangQianQian2004}.

This path-space asymmetry plays a significant role in several areas.
In stochastic thermodynamics, entropy production underlies fluctuation relations and, with the appropriate physical time-reversal convention, quantifies dissipation \citep{Seifert2012}.
It also provides a basis for inferring dissipation from stationary observations \citep{RoldanParrondo2010} and for assessing the irreversibility introduced or distorted by numerical discretization \citep{KatsoulakisPantazisReyBellet2014,DaCostaPavliotis2023}.
Throughout this paper, time reversal acts only on the temporal order of the state path, without an additional state-space involution such as momentum reversal.

In this work, we explore whether the path-space formulation of EPR can extend to infinite-dimensional stochastic systems.
This question is natural for stochastic field theories, where the evolving state is a spatial field rather than a finite vector \citep{DaPratoZabczyk2014} and entropy production has been used to characterize irreversibility in, for example, active matter and stochastic hydrodynamic fields \citep{Nardini2017,ComplexFluidsDiscretisation2022}.
Meanwhile, our recent work \citep{ZhaoZhangLi2024,zhu2026identifiable} quantifies EPR from data within finite-dimensional stochastic dynamics, while a separate line of work \citep{li2026hypothesis} leads us to mesoscopic stochastic dynamics whose states belong naturally to function spaces.
The passage to infinite dimensions, however, is not merely a formal extension of the finite-dimensional theory.
Path-space relative entropy decreases under finite-dimensional projection \citep{Gray2011}; hence, finiteness at every finite-dimensional resolution does not preclude divergence as increasingly many modes are resolved.
A formulation based solely on finite-dimensional truncations is therefore insufficient, and EPR must instead be defined and analyzed directly on the full Hilbert-space path space.

We consider stationary diffusions on a separable Hilbert space $\mH$ governed by
\begin{equation}\label{eq:intro-hilbert-diffusion}
    \md X_t=F(X_t)\md t+B(X_t)\md W_t, \qquad a(x)=B(x)B(x)^*,
\end{equation}
where $W$ is cylindrical on a separable Hilbert space $\mU$, $B(x)\colon\mU\to\mH$ is Hilbert--Schmidt, and $\mu$ is an invariant probability measure.
The trace-class covariance $a(x)=B(x)B(x)^*$ may be degenerate and state dependent.
Our aim is to determine how the irreversible part of the drift, measured in the geometry induced by this noise, controls the full path-space entropy production.

Under the integrability and integration-by-parts assumptions of \cref{sec:diffusion-framework}, we obtain the reversible--irreversible decomposition
\begin{equation}\label{eq:intro-drift-decomposition}
    F=\Frev+\Firr, \qquad \Frev=\tfrac12 d, \quad \Firr=F-\tfrac12d,
\end{equation}
where $d$ is the drift correction characterized by the covariance $a$ with respect to $\mu$.
The irreversible drift is divergence-free with respect to $\mu$.
The corresponding formal candidate for the reversed drift is $\bar F=\Frev-\Firr$, and its identification with the actual stationary reversal will be discussed separately.

The connection between this decomposition and entropy production is explicit in finite dimensions.
Under suitable regularity and time-reversal hypotheses, \citet[Theorem~4.1]{DaCostaPavliotis2023} establish the formula
\begin{equation}\label{eq:intro-finite-dimensional-entropy}
    \epr =2\int_{\mathbb R^n} \norm{a(x)^{-1/2}\Firr(x)}_{\mathbb R^n}^{2} \mu(\md x)
\end{equation}
for finite-dimensional cases, provided that $\Firr(x)\in\Ran(a(x)^{1/2})$ for $\mu$-almost every $x$ and the integral is finite.
On one hand, the range condition ensures that the drift difference $\bar F-F=-2\Firr$ can be represented by a shift of the driving noise.
In finite dimensions, it simply excludes an irreversible drift component in directions with zero instantaneous noise variance; in particular, it is automatic when the covariance is invertible.
The integrability condition controls the stationary average of the required noise energy.
On the other hand, the inverse covariance gives greater weight to a given drift component when the noise is weaker in that direction.
Under the same hypotheses, the forward and reversed path laws are equivalent on every finite time interval.

However, in infinite dimensions, the absence of noise-free directions no longer guarantees that an irreversible drift admits a finite-energy representation through the noise.
This obstruction appears already for a constant, injective trace-class covariance $a$ on an infinite-dimensional Hilbert space, for which the range $\Ran(a^{1/2})$ is dense but not closed in $\mH$, and $a^{-1/2}$ is unbounded.
Under such circumstances, $\Firr(x)\notin\Ran(a^{1/2})$ may result from the accumulation of contributions over infinitely many weakly forced directions, rather than from a drift component in a noise-free direction.
Even when the pointwise energy is finite almost everywhere, its stationary average still requires separate control so that the integral is finite.
Moreover, the finite-dimensional identity $\Ran(a)=\Ran(a^{1/2})$ need not persist, so requiring membership in the covariance range would impose a stronger condition than finite noise energy.
The path-law comparison also requires a separate justification: a formal reversed drift, or its identification in each coordinate, does not by itself provide an $\mH$-valued reversed diffusion with an integrable drift.
Therefore, the finite-dimensional formula identifies a natural candidate energy, but its relation to the full path-space EPR, including when that energy is infinite, requires an independent argument.

A separate difficulty is probabilistic.
An integration-by-parts identity on cylinder functions identifies a candidate reversed drift, but does not by itself identify the generator or stochastic equation of the stationary time reversal.
Infinite-dimensional time-reversal results of \citet{FollmerWakolbinger1986} and \citet{MilletNualartSanzInfinite1989} address this issue under locally finite-entropy or analytic hypotheses.
In particular, identifying the reversed dynamics coordinatewise does not automatically yield an $\mH$-valued equation with a Bochner-integrable drift.
In summary, these two difficulties play different roles in the entropy calculation: reversal identification permits a change of measure between the forward and reversed equations, whereas the Cameron--Martin condition determines whether the drift difference can be realized by a finite-energy shift of the noise.

We treat the identification of the stationary reversal and the entropy calculation as separate problems.
Under \cref{asm:stationary-diffusion,asm:cylinder-generator-integrability,asm:covariance-divergence}, an antisymmetric two-point test and the forward It\^o formula give a lower bound for the EPR in terms of the extended Cameron--Martin energy:
\begin{equation}\label{eq:intro-entropy-lower-bound}
    \epr\ge2\int_{\mH}g^2\md\mu,\qquad g(x)=
    \begin{cases}
        \norm{a(x)^{-1/2}\Firr(x)}_{\mH},&\Firr(x)\in\Ran\left(a(x)^{1/2}\right),\\
        +\infty,&\text{otherwise}.
    \end{cases}
\end{equation}
The estimate requires neither a diffusion representation of the stationary reversal nor absolute continuity between the forward and reversed path laws.
Consequently, if the integral is infinite, which means either the range condition fails on a set of positive $\mu$-measure, or the Cameron--Martin energy is not integrable, then we have an infinite EPR.
In the regular regime when the integral is finite, it can be shown that $u=2B^+\Firr$ is the corresponding noise control, and
\begin{equation}\label{eq:intro-finite-energy-identity}
    \mP_{[0,T]}\sim\mP^{\mathrm{rev}}_{[0,T]}, \qquad \epr=2\int_{\mH}g^2\md\mu<+\infty
\end{equation}
when the time-reversal assumption holds, the Girsanov exponential associated with $-u(X_t)$ is a true martingale, and the reversed equation is unique in law from $\mu$.
A sufficient condition involving Lipschitz continuity of the forward or reversed drift and the noise coefficient is provided.

To verify the time-reversal assumption, we establish two criteria based on the abstract integration-by-parts formula of \citet{CCGL2023}.
For a constant diffusion coefficient, we identify the reversed dynamics coordinatewise and characterize the additional path integrability needed for an $\mH$-valued realization.
This integrability follows from finite Cameron--Martin energy.
For a continuous state-dependent diffusion coefficient, we work directly on $\mH$ and extend the invariant-measure integration-by-parts identity beyond cylinder functions.
Under finite Cameron--Martin energy and the stated norm-integrability assumptions, this yields the reversed equation with drift $\bar F=d-F$ and diffusion coefficient $B$.

Together, the lower bound and the finite-energy identity yield verifiable conditions for finite and infinite entropy production.
We illustrate them with two nonlinear infinite-rank diffusions on $\ell^2(\mathbb N;\mathbb R^2)$.
The first has constant covariance and a common radial coefficient coupling infinitely many rotating blocks.
It preserves a Gaussian measure, has globally Lipschitz drift, and, for nontrivial choices of the rotational coefficients, has strictly positive finite EPR.
Moreover, the parameters can be chosen so that $\Firr\in\Ran(a^{1/2})\setminus\Ran(a)$ for $\mu$-a.e. $x$, showing that finite entropy production only requires finite energy in the local Cameron--Martin geometry and need not impose membership in the smaller covariance range.
The second example has state-dependent blockwise covariance and again preserves a Gaussian measure.
For $\theta\in(0,1)$, the irreversible drift lies outside the noise range for $\mu$-almost every state, so the lower bound gives infinite full path-space entropy production.
Nevertheless, every truncation to the first $N$ two-dimensional blocks has finite EPR.
Together, the two examples exhibit two specifically infinite-dimensional effects: the distinction between the covariance and noise ranges, and the failure of finiteness at every fixed finite-dimensional resolution to ensure finite entropy production for the full process.

A closely related work is \citet{DuanZhangZimmer2026}, which derives an entropy-production formula and reversibility criteria relative to a reversible Gaussian Ornstein--Uhlenbeck reference process.
Both approaches identify the relevant cost through Cameron--Martin energy, but the frameworks differ in their assumptions and scope.
Their Assumption~2.7 imposes smoothness, an $\mH$-valued strong realization, and a covariance-range factorization of the nonlinear perturbation.
Our calculus is formulated directly relative to the invariant measure $\mu$, allows state-dependent covariance, and yields an extended-valued lower bound that also detects failure of the noise-range condition or non-integrability of the corresponding energy.
In particular, for the parameter choice above with $\Firr\in\Ran(a^{1/2})\setminus\Ran(a)$, the first infinite-rank example falls outside the covariance-range factorization associated with its natural reversible Ornstein--Uhlenbeck reference.
These are distinctions between the two sets of hypotheses rather than an inclusion of one theory in the other.

The paper is organized as follows.
\Cref{sec:diffusion-framework} introduces the stationary Hilbert-space diffusion, the invariant-measure calculus and reversible--irreversible decomposition, path-space entropy, and Girsanov change of measure.
\Cref{sec:entropy-production-bounds} establishes the extended Cameron--Martin lower bound and proves its sharpness in the regular regime under explicit additional conditions.
\Cref{sec:time-reversal} develops two criteria for identifying the stationary reversed diffusion, for constant and state-dependent diffusion coefficients, respectively, and we have provided two concrete nonlinear infinite-rank examples separately in \cref{sec:consequences-and-examples}.
\Cref{sec:conclusion} discusses the scope of the results and possible extensions.

\section{The infinite-dimensional EPR framework}\label{sec:diffusion-framework}

This section develops the analytic and probabilistic framework used throughout the paper.
We first specify the stationary diffusion and define its path-space entropy production.
We then develop an $L^1(\mu)$ generator calculus that identifies the formal adjoint and the reversible--irreversible decomposition of the drift.
The section concludes with Girsanov's theorem and the corresponding density formula for the induced state-path laws.

\subsection{Stationary diffusions and entropy production}
Let $\mH$ and $\mU$ be real separable Hilbert spaces, and let $(W_t)_{t\ge0}$ be a cylindrical Wiener process on $\mU$.
Throughout the paper, we consider
\begin{equation}\label{eq:stationary-diffusion}
    \md X_t=F(X_t)\md t+B(X_t)\md W_t, \qquad X_t\in\mH,
\end{equation}
where the drift $F\colon\mH\to\mH$ and the diffusion coefficient $B\colon\mH\to\mathcal L_2(\mU,\mH)$ are Borel measurable.
Here $\mathcal L_2(\mU,\mH)$ denotes the space of Hilbert--Schmidt operators from $\mU$ to $\mH$.
We write $a(x)\coloneqq B(x)B(x)^*$ for the diffusion covariance.
We work with a stationary weak solution satisfying the following assumption.
\begin{assumption}\label{asm:stationary-diffusion}
\Cref{eq:stationary-diffusion} admits a global weak solution $X$ that is a $\mu$-stationary, time-homogeneous Markov process with continuous paths, where $\mu$ is a Borel probability measure on $\mH$.
\end{assumption}

Here a weak solution is understood in the probabilistic sense on a filtered probability space $(\Omega,\mathcal F,(\mathcal F_t)_{t\ge0},\mathbb P)$ satisfying the usual conditions.
The process $X$ is adapted, $W$ is a cylindrical Wiener process relative to $(\mathcal F_t)$, and, for every $T>0$,
\begin{equation}\label{eq:weak-solution-integrability}
    \int_0^T\left( \|F(X_s)\|_{\mH} +\|B(X_s)\|_{\mathcal L_2(\mU,\mH)}^2 \right)\md s<+\infty \qquad\mathbb P\text{-a.s.}
\end{equation}
The integral form of \cref{eq:stationary-diffusion} holds in $\mH$, with the drift integral understood in the Bochner sense.

For time horizon $T>0$, let $\Prob_{[0,T]}$ denote the law of $X|_{[0,T]}$ on $C([0,T];\mH)$.
Define the time-reversal map as $\mathfrak r_T\omega\coloneqq(\omega_{T-t})_{0\le t\le T}$ and the reversed path law as $\Prob^{\mathrm{rev}}_{[0,T]}\coloneqq(\mathfrak r_T)_\#\Prob_{[0,T]}$.
When the time horizon is fixed, we suppress the subscript $[0,T]$ if no confusion arises.

To help evaluate the EPR value, it can be shown that for stationary time-homogeneous Markov processes on a Polish space, the path-space relative entropy is exactly linear in the time horizon, and thus the limit is superfluous in the definition of EPR, see, e.g., \citet[Theorem~ 2.2.4]{JiangQianQian2004} and \citet[Theorem~10.4]{varadhan2006large}.

\begin{proposition}\label{prop:entropy-time-linearity}
Under \cref{asm:stationary-diffusion}, there exists a constant $c\in[0,+\infty]$ such that the relative entropy between the forward and time-reversal path measures
\begin{equation}\label{eq:entropy-time-linearity}
    \KLdiv{\Prob_{[0,T]}}{\Prob^{\mathrm{rev}}_{[0,T]}}=cT,\qquad\forall T>0.
\end{equation}
\end{proposition}
The linearity yields the following general definition of the entropy production rate for stationary Markov processes:
\begin{equation}\label{eq:entropy-production-rate}
    \epr\coloneqq\frac1T\KLdiv{\Prob_{[0,T]}}{\Prob^{\mathrm{rev}}_{[0,T]}},
\end{equation}
which is independent of $T$ and agrees with the long-time definition.

Our aim is to characterize the finiteness of $\epr$ and relate its value to the drift, the invariant measure, and the geometry of the noise.
We now introduce the generator framework needed for this analysis.
We use smooth bounded cylinder functions as the basic test class.
\begin{definition}
The class of smooth bounded cylinder functions is defined as
\begin{equation}
    \begin{aligned}
        \FC\coloneqq\bigcup_{n\ge1}\bigl\{\varphi\colon\mH\to\R\mid\varphi(x)
        =f\left(\ip{x}{h_1}_\mH,\ldots,\ip{x}{h_n}_\mH\right)\;\text{for}\;f\in C_b^\infty(\R^n),\, h_1,\ldots,h_n\in\mH\bigr\}.
    \end{aligned}
\end{equation}
Here $C_b^\infty(\R^n)$ consists of bounded smooth functions with bounded derivatives of every order.
\end{definition}

Let $(P_t)_{t\ge0}$ denote the transition semigroup of the stationary process $X$.
Since $\mH$ is Polish, we may choose Borel versions of the transition kernels and write $(P_tu)(x)\coloneqq\E[u(X_t)\mid X_0=x]$ for any bounded Borel $u$ with the identity understood $\mu$-almost everywhere.
On the test-function class $\FC$, we define the differential operator
\begin{equation}\label{eq:cylinder-generator}
    L\varphi\coloneqq \ip{F}{D\varphi}_{\mH} +\frac12\Tr(aD^2\varphi), \qquad\forall\varphi\in\FC.
\end{equation}
Here $D\varphi$ and $D^2\varphi$ denote the Fr\'echet gradient and Hessian of $\varphi$, respectively.
The trace operation is well defined since $a(x)$ is trace class and $D^2\varphi$ is a bounded finite-rank operator.

\begin{assumption}\label{asm:cylinder-generator-integrability}
The operator $L$ satisfies
\begin{equation}
    L\varphi\in L^1(\mu),\quad\forall\varphi\in\FC.
\end{equation}
\end{assumption}

The following equivalent formulation will be useful when verifying \cref{asm:cylinder-generator-integrability} in applications.

\begin{lemma}\label{lem:generator-integrability-characterization}
\Cref{asm:cylinder-generator-integrability} is equivalent to
\begin{equation}\label{eq:generator-integrability-characterization}
    \ip{F}{h}_\mH\in L^1(\mu),\quad\forall h\in\mH;\qquad\ip{ah_1}{h_2}_\mH\in L^1(\mu),\quad\forall h_1,h_2\in\mH.
\end{equation}
\end{lemma}
\begin{proof}
For any $h\in\mH$, we write $x_h\coloneqq\ip{x}{h}_{\mH}$.
By \cref{asm:cylinder-generator-integrability} we have $L(\sin x_h),L(\cos x_h)\in L^1(\mu)$, and a direct computation yields
\begin{equation}
    \begin{aligned}
        \ip{F}{h}_\mH&=(\cos x_h)L(\sin x_h)-(\sin x_h)L(\cos x_h)\in L^1(\mu),\\
        \ip{ah}{h}_\mH&=-2(\sin x_h)L(\sin x_h)-2(\cos x_h)L(\cos x_h)\in L^1(\mu).
    \end{aligned}
\end{equation}
The integrability for $\ip{ah_1}{h_2}_\mH$ follows by polarization.
Conversely, if \cref{eq:generator-integrability-characterization} holds, $L\varphi\in L^1(\mu)$ since it can be decomposed into a finite sum of inner products in $L^1(\mu)$ with bounded coefficients.
\end{proof}

Note that these coordinatewise conditions do not require $F\in L^1(\mu;\mH)$ or $B\in L^2(\mu;\HS(\mU,\mH))$.
Under stationarity, these stronger conditions are sufficient for both \cref{eq:weak-solution-integrability,asm:cylinder-generator-integrability}.
They will be discussed later when needed.

We next identify the differential expression \cref{eq:cylinder-generator} with the $L^1(\mu)$ generator of the stationary Markov semigroup.

\begin{proposition}\label{prop:forward-generator-identification}
Under \cref{asm:stationary-diffusion,asm:cylinder-generator-integrability}, $(P_t)_{t\ge0}$ extends to a strongly continuous contraction semigroup on $L^1(\mu)$.
Denote its generator by $\tilde L$.
Then $\tilde L$ is well defined on $\FC$ and
\begin{equation}\label{eq:forward-generator-identification}
    \tilde{L}\varphi=L\varphi,\qquad\forall\varphi\in\FC,
\end{equation}
where the cylinder functions are identified modulo $\mu$-null sets.
Moreover, for every $t\ge0$ and $\varphi\in\FC$,
\begin{equation}\label{eq:forward-semigroup-integral}
    P_t\varphi-\varphi =\int_0^tP_s(L\varphi)\md s \qquad\text{in }L^1(\mu),
\end{equation}
where the time integral is a Bochner integral in $L^1(\mu)$.
\end{proposition}

\begin{proof}
Stationarity and conditional Jensen give $\|P_tv\|_{L^1(\mu)}\le\|v\|_{L^1(\mu)}$ for bounded Borel $v$, so $(P_t)$ extends to a contraction semigroup on $L^1(\mu)$.
For $v\in C_b(\mH)$, path continuity and dominated convergence yield
\begin{equation}\label{eq:forward-semigroup-continuity}
    \|P_tv-v\|_{L^1(\mu)} \le\E\bigl[|v(X_t)-v(X_0)|\bigr] \to0 \quad\text{as}\quad t\to0^+.
\end{equation}
Since $C_b(\mH)$ is dense in $L^1(\mu)$, the contraction property leads to the strong continuity of $(P_t)_{t\ge0}$ on $L^1(\mu)$.

Fix $\varphi\in\FC$.
By It\^o's formula, the process
\begin{equation}\label{eq:cylinder-function-martingale}
    M_t^\varphi\coloneqq\varphi(X_t)-\varphi(X_0)-\int_0^tL\varphi(X_s)\md s=\int_0^t\ip{B(X_s)^*D\varphi(X_s)}{\md W_s}_\mU
\end{equation}
is a continuous local martingale.
For every $T>0$, stationarity and \cref{asm:cylinder-generator-integrability} give an upper bound
\begin{equation}\label{eq:cylinder-martingale-integrability}
    \E\sup_{0\le t\le T}\left|M_t^\varphi\right| \le2\|\varphi\|_\infty +\E\int_0^T|L\varphi(X_s)|\md s =2\|\varphi\|_\infty+T\|L\varphi\|_{L^1(\mu)} <+\infty,
\end{equation}
and thus $M^\varphi$ is a true martingale on $[0,T]$.
Taking conditional expectations given $X_0$ and applying conditional Fubini's theorem to \cref{eq:cylinder-function-martingale} yields \cref{eq:forward-semigroup-integral}.
Finally, strong continuity gives
\begin{equation}\label{eq:forward-generator-limit}
    \left\| \frac{P_t\varphi-\varphi}{t}-L\varphi \right\|_{L^1(\mu)} \le\frac1t\int_0^t \norm{P_s(L\varphi)-L\varphi}_{L^1(\mu)}\md s \to0 \quad\text{as}\quad t\to0^+.
\end{equation}
This proves \cref{eq:forward-generator-identification} by the definition of the $L^1(\mu)$ generator.
\end{proof}

Therefore, in what follows, we retain $L$ for the differential operator on $\FC$, where it agrees with the forward generator $\tilde L$.

\subsection{Invariant-measure calculus and formal adjoint}\label{sec:invariant-measure-calculus}
We now use stationarity and integration by parts with respect to $\mu$ to identify the formal adjoint of $L$ on $\FC$.
The resulting operator has the same diffusion covariance as $L$, with a drift determined by the divergence of the covariance directions.

To start with, we define the carr\'e du champ $\Gamma(\phi,\psi)\coloneqq\ip{aD\phi}{D\psi}_\mH$ for any $\phi,\psi\in\FC$.
\begin{proposition}\label{prop:stationary-generator-calculus}
We have the identity
\begin{equation}\label{eq:generator-product-rule}
    L(\phi\psi)=\phi L\psi+\psi L\phi+\Gamma(\phi,\psi),\qquad\forall\phi,\psi\in\FC.
\end{equation}
Furthermore, under \cref{asm:stationary-diffusion,asm:cylinder-generator-integrability}, we have
\begin{equation}\label{eq:stationary-generator-mean}
    \int_\mH L\varphi\md\mu=0,\qquad\forall\varphi\in\FC,
\end{equation}
and it follows that
\begin{equation}\label{eq:carre-du-champ-integrability}
    \int_\mH\Gamma(\varphi,\varphi)\md\mu =-2\int_\mH\varphi L\varphi\md\mu \le 2\norm{\varphi}_\infty\norm{L\varphi}_{L^1(\mu)}<+\infty,\quad\forall\varphi\in\FC.
\end{equation}
Consequently, $\Gamma(\phi,\psi)\in L^1(\mu)$ for every $\phi,\psi\in\FC$.
\end{proposition}
\begin{proof}
The identity \cref{eq:generator-product-rule} can be checked by direct calculation.
Fix $t>0$.
Integrating \cref{eq:forward-semigroup-integral} against $\mu$ and using invariance yields
\begin{equation}
    0 =\int_{\mH}(P_t\varphi-\varphi)\md\mu =\int_0^t\int_{\mH}P_s(L\varphi)\md\mu\md s =t\int_{\mH}L\varphi\md\mu,
\end{equation}
which proves \cref{eq:stationary-generator-mean}.
Integrating \cref{eq:generator-product-rule} with $\psi=\phi=\varphi$ and using stationarity implies \cref{eq:carre-du-champ-integrability}.
Finally, the integrability of $\Gamma(\phi,\psi)$ follows by the polarization identity.
\end{proof}

\begin{definition}\label{def:invariant-measure-divergence}
A Borel measurable map $V\colon\mH\to\mH$ admits a $\mu$-divergence if
\begin{equation}
    \ip{V}{D\varphi}_\mH\in L^1(\mu),\qquad\forall\varphi\in\FC,
\end{equation}
and there exists $\divg_\mu V\in L^1(\mu)$ satisfying
\begin{equation}\label{eq:divergence-integration-by-parts}
    \int_\mH\ip{V}{D\varphi}_\mH\md\mu =-\int_\mH(\divg_\mu V)\varphi\md\mu, \qquad\forall\varphi\in\FC.
\end{equation}
\end{definition}
To represent the divergences of the covariance directions by a single $\mH$-valued field, we impose the following assumption.
\begin{assumption}\label{asm:covariance-divergence}
There exists a Borel measurable map $d\colon\mH\to\mH$ such that, for any $h\in\mH$, the mapping $x\mapsto a(x)h$ admits a $\mu$-divergence and
\begin{equation}\label{eq:covariance-direction-divergence}
    \ip{d}{h}_\mH=\divg_\mu(ah)\in L^1(\mu).
\end{equation}
Equivalently, for any $h\in\mH$, we have $\ip{d}{h}_\mH\in L^1(\mu)$ and
\begin{equation}
    \int_\mH\ip{d}{h}_\mH\varphi\md\mu=-\int_\mH\ip{ah}{D\varphi}_\mH\md\mu,\qquad\forall\varphi\in\FC.
\end{equation}
\end{assumption}
\Cref{asm:covariance-divergence} requires the divergences of all covariance directions $ah$ to be represented by a single $\mH$-valued field $d$.
In finite dimensions, for a smooth positive invariant density $\rho^\fin$ and a smooth covariance $a^\fin$, the corresponding identity reads $d^\fin=(\rho^\fin)^{-1}\nabla\cdot(\rho^\fin a^\fin)$.
\Cref{asm:covariance-divergence} formulates this integration-by-parts structure directly with respect to $\mu$, without requiring an ambient density.
\begin{proposition}
For any Borel measurable map $V:\mH\to\mH$ with a $\mu$-divergence, we have
\begin{equation}\label{eq:divergence-product-rule}
    \divg_\mu(\varphi V)=\varphi\divg_\mu V+\ip{V}{D\varphi}_\mH,\qquad\forall\varphi\in\FC.
\end{equation}
Moreover, with \cref{asm:covariance-divergence},
\begin{equation}\label{eq:covariance-gradient-divergence}
    \divg_\mu(aD\varphi) =\ip{d}{D\varphi}_\mH+\Tr(aD^2\varphi),\qquad\forall\varphi\in\FC.
\end{equation}
\end{proposition}
\begin{proof}
For any $\phi,\psi\in\FC$,
\begin{equation}
    \begin{aligned}
        \int_\mH\ip{\phi V}{D\psi}_\mH\md\mu&=\int_\mH\ip{V}{D(\phi\psi)}_\mH\md\mu-\int_\mH\psi\ip{V}{D\phi}_\mH\md\mu
        =-\int_\mH\psi\left(\phi\divg_\mu V+\ip{V}{D\phi}_\mH\right)\md\mu,
    \end{aligned}
\end{equation}
which implies \cref{eq:divergence-product-rule}.
For $\varphi=f(\ip{\cdot}{h_1}_\mH,\cdots,\ip{\cdot}{h_n}_\mH)\in\FC$, we have $aD\varphi=\sum_i(\partial_if)ah_i$, and it follows that
\begin{equation}
    \begin{aligned}
        \divg_\mu(aD\varphi)&=\sum_i\left[(\partial_if)\divg_\mu(ah_i)+\ip{ah_i}{D(\partial_if)}_\mH\right]\\
        &=\sum_i\left[(\partial_if)\ip{d}{h_i}_\mH+\sum_j\ip{ah_i}{(\partial_{ij}f)h_j}_\mH\right]\\
        &=\ip{d}{D\varphi}_\mH+\Tr(aD^2\varphi)
    \end{aligned}
\end{equation}
by applying \cref{eq:divergence-product-rule} to each summand.
\end{proof}

We can now identify the formal adjoint of the cylinder operator $L$ with respect to $\mu$.
\begin{proposition}
Under \cref{asm:stationary-diffusion,asm:cylinder-generator-integrability,asm:covariance-divergence}, define
\begin{equation}\label{eq:formal-adjoint-generator}
    {\bar L}\varphi\coloneqq\divg_\mu (aD\varphi)-L\varphi=\ip{d-F}{D\varphi}_\mH+\frac12\Tr\left(aD^2\varphi\right),\quad\forall\varphi\in\FC.
\end{equation}
Then ${\bar L}\varphi\in L^1(\mu)$ and
\begin{equation}\label{eq:generator-adjoint-duality}
    \int_\mH\psi(L\phi)\md\mu=\int_\mH\phi({\bar L}\psi)\md\mu,\quad\forall\phi,\psi\in\FC.
\end{equation}
\end{proposition}
\begin{proof}
Since $\divg_\mu(aD\varphi)\in L^1(\mu)$ by \cref{eq:covariance-gradient-divergence}, the integrability of ${\bar L}\varphi$ follows from \cref{asm:cylinder-generator-integrability} and \cref{eq:formal-adjoint-generator}.
By stationarity and \cref{eq:generator-product-rule}, we have
\begin{equation}
    \int_\mH\psi(L\phi)\md\mu=-\int_\mH\Gamma(\phi,\psi)\md\mu-\int_\mH\phi L\psi\md\mu=\int_\mH\phi\left(\divg_\mu (aD\psi)-L\psi\right)\md\mu,\quad\forall\psi,\phi\in\FC,
\end{equation}
which verifies \cref{eq:generator-adjoint-duality} by \cref{eq:formal-adjoint-generator}.
\end{proof}
The term ``formal adjoint'' refers only to the duality \cref{eq:generator-adjoint-duality} on the cylinder class $\FC$.
The identity does not alone identify ${\bar L}$ with the $L^1(\mu)$ generator of the stationary reversed semigroup, nor does it establish a reversed stochastic equation.
These questions will be addressed separately in \cref{sec:time-reversal}.

\subsection{Reversible--irreversible decomposition}\label{sec:drift-decomposition}
The formal adjoint separates $L$ into formally symmetric and antisymmetric parts with respect to $\mu$ on $\FC$.
This induces the corresponding reversible--irreversible decomposition of the drift.

\begin{proposition}\label{prop:reversible-irreversible-decomposition}
Under \cref{asm:stationary-diffusion,asm:cylinder-generator-integrability,asm:covariance-divergence}, we set
\begin{equation}\label{eq:drift-decomposition}
    \Frev\coloneqq d/2,\qquad\Firr\coloneqq F-d/2
\end{equation}
and define
\begin{equation}\label{eq:reversible-irreversible-generators}
    \begin{aligned}
        L_\rev\varphi&\coloneqq\frac12(L+{\bar L})\varphi=\frac12\ip{d}{D\varphi}_\mH+\frac12\Tr(aD^2\varphi),\\
        L_\irr\varphi&\coloneqq\frac12(L-{\bar L})\varphi=\ip{F-\frac12d}{D\varphi}_\mH
    \end{aligned}
\end{equation}
for any $\varphi\in\FC$.
Then we have
\begin{equation}\label{eq:generator-decomposition}
    L=L_\rev+L_\irr,\qquad {\bar L}=L_\rev-L_\irr.
\end{equation}
Moreover, $L_\rev$ and $L_\irr$ are formally symmetric and antisymmetric on $\FC$, respectively, which means
\begin{equation}
    \int_\mH\psi(L_\rev\phi)\md\mu=\int_\mH\phi(L_\rev\psi)\md\mu,\qquad\int_\mH\psi(L_\irr\phi)\md\mu=-\int_\mH\phi(L_\irr\psi)\md\mu,\qquad\forall\psi,\phi\in\FC.
\end{equation}
The reversible part $L_\rev$ is in $\mu$-divergence form, whereas the irreversible drift satisfies $\divg_\mu\Firr=0$.
In particular,
\begin{equation}\label{eq:decomposed-generator-means}
    \int_\mH L_\rev\varphi\md\mu=\int_\mH L_\irr\varphi\md\mu=0,\quad\forall\varphi\in\FC.
\end{equation}
\end{proposition}

\begin{proof}
\Cref{eq:reversible-irreversible-generators,eq:generator-decomposition} follow directly from \cref{eq:cylinder-generator,eq:formal-adjoint-generator}.
The divergence identity \cref{eq:divergence-integration-by-parts} results in
\begin{equation}
    \int_\mH\psi (L_\rev\phi)\md\mu=-\frac12\int_\mH\Gamma(\phi,\psi)\md\mu=\int_\mH\phi (L_\rev\psi)\md\mu,
\end{equation}
and the formal antisymmetry of $L_\irr$ follows from \cref{eq:generator-adjoint-duality} by taking the half-difference of $L$ and ${\bar L}$.
The definition of $\mu$-divergence together with \cref{eq:stationary-generator-mean} implies \cref{eq:decomposed-generator-means}, and thus $\divg_\mu F_\irr=0$ by definition.
\end{proof}
The identity $\divg_\mu\Firr=0$ is the invariant-measure analogue of the stationary-current condition $\nabla\cdot(\rho^\fin F_\irr^\fin)=0$ in finite dimensions.
Thus $\Firr$ plays the role of the stationary current velocity without requiring a reference Lebesgue measure on $\mH$.
The formal adjoint preserves the reversible drift and changes the sign of the irreversible drift.
Accordingly, we define the candidate reversed drift by
\begin{equation}\label{eq:candidate-reversed-drift}
    \bar F=d-F=F-2\Firr,
\end{equation}
which is naturally a candidate drift of the stationary time-reversed diffusion.
Taking \cref{eq:cylinder-generator,eq:formal-adjoint-generator,eq:reversible-irreversible-generators,eq:candidate-reversed-drift,eq:drift-decomposition} together, we have
\begin{equation}
    \begin{aligned}
        L\varphi&=\langle\underbrace{F_\rev+F_\irr}_{=F},{D\varphi}\rangle_\mH+\frac12\Tr(aD^2\varphi)& {\bar L}\varphi
        &=\langle{\underbrace{F_\rev-F_\irr}_{=\bar F}},{D\varphi}\rangle_\mH+\frac12\Tr(aD^2\varphi),\\
        &=L_\rev\varphi+L_\irr\varphi,&&=L_\rev\varphi-L_\irr\varphi,
    \end{aligned}
\end{equation}
where
\begin{equation}
    L_\rev\varphi=\ip{F_\rev}{D\varphi}_\mH+\frac12\Tr(aD^2\varphi),\qquad L_\irr\varphi=\ip{F_\irr}{D\varphi}_\mH
\end{equation}
for any $\varphi\in\FC$.

\subsection{Girsanov change of measure}\label{sec:girsanov-change-of-measure}
Fix $T>0$ and work on the filtered probability space supporting the stationary weak solution in \cref{asm:stationary-diffusion}.
Let $(\psi_t)_{0\le t\le T}$ be a $\mU$-valued predictable process satisfying $\int_0^T\norm{\psi_s}_{\mU}^2\md s<+\infty$ $\Prob$-almost surely.
Define the Dol\'eans--Dade exponential
\begin{equation}\label{eq:girsanov-exponential}
    Z_t^\psi \coloneqq \exp\left( \int_0^t\ip{\psi_s}{\md W_s}_{\mU} -\frac12\int_0^t\norm{\psi_s}_{\mU}^2\md s \right), \qquad 0\le t\le T.
\end{equation}
It is a strictly positive local martingale and hence a supermartingale.
When it is a true martingale, Girsanov's theorem changes the drift of $X$ by $B(X_t)\psi_t$.
We record this standard result \citep[Ch.~10]{DaPratoZabczyk2014}, together with the density formula for the induced state-path law \citep[Proposition~5(1)]{Leonard2014}.

\begin{theorem}\label{thm:girsanov-change-of-measure}
In the above setting, suppose that $(Z_t^\psi)_{0\le t\le T}$ is a true $\Prob$-martingale.
Define a probability measure $\widehat\Prob_T^\psi$ on $\mathcal F_T$ by $\md\widehat\Prob_T^\psi=Z_T^\psi\md\Prob$.
Then $\widehat\Prob_T^\psi\sim\Prob$ on $\mathcal F_T$, and the two measures agree on $\mathcal F_0$.
The process $\widehat W$, defined cylindrically by
\begin{equation}\label{eq:girsanov-shifted-wiener}
    \widehat W_t(h) \coloneqq W_t(h)-\int_0^t\ip{\psi_s}{h}_{\mU}\md s, \qquad h\in\mU,\quad 0\le t\le T,
\end{equation}
is a cylindrical Wiener process on $\mU$ relative to $(\mathcal F_t)_{0\le t\le T}$ under $\widehat\Prob_T^\psi$.
Under this measure, the same process $X$ satisfies
\begin{equation}\label{eq:girsanov-shifted-diffusion}
    \md X_t = \bigl(F(X_t)+B(X_t)\psi_t\bigr)\md t +B(X_t)\md\widehat W_t, \qquad 0\le t\le T,
\end{equation}
in the integral sense in $\mH$, with unchanged initial distribution $\mu$.

Let $\mathcal F_T^X$ be the $\Prob$-completion of $\sigma(X_s:0\le s\le T)$, and denote the induced state-path law by
\begin{equation}\label{eq:girsanov-tilted-path-law}
    \widehat\Prob_{[0,T]}^\psi \coloneqq \widehat\Prob_T^\psi \circ\bigl(X|_{[0,T]}\bigr)^{-1}.
\end{equation}
Then $\widehat\Prob_{[0,T]}^\psi\sim\Prob_{[0,T]}$.
Its Radon--Nikodym derivative admits a Borel version $\rho_T^\psi\colon C([0,T];\mH)\to(0,+\infty)$ satisfying
\begin{equation}\label{eq:girsanov-path-density}
    \begin{aligned}
        \frac{\md\widehat\Prob_{[0,T]}^\psi} {\md\Prob_{[0,T]}} \bigl(X|_{[0,T]}\bigr)=\rho_T^\psi\bigl(X|_{[0,T]}\bigr)=\E_\Prob\bigl[Z_T^\psi\mid\mathcal F_T^X\bigr]
        \qquad\Prob\text{-a.s.}
    \end{aligned}
\end{equation}
\end{theorem}

\section{Entropy-production bounds and the finite-energy regime}\label{sec:entropy-production-bounds}
In \cref{sec:drift-decomposition}, we have introduced the mapping $\bar F=F-2\Firr$ as the candidate reversed drift.
The lower bound below uses only \cref{asm:stationary-diffusion,asm:cylinder-generator-integrability,asm:covariance-divergence}.
For the matching upper bounds in the regular regime, we will also assume that the genuine stationary reversal admits the following diffusion representation, which will be verified in \cref{sec:time-reversal}.

\begin{assumption}\label{asm:reversed-diffusion}
The stationary time reversal of \cref{eq:stationary-diffusion} is an $\mH$-valued It\^o diffusion
\begin{equation}\label{eq:stationary-reversed-diffusion}
    \md\bar X_t =\bar F(\bar X_t)\md t+B(\bar X_t)\md\widetilde W_t, \qquad \bar F=F-2\Firr,
\end{equation}
with a reverse-time cylindrical Wiener process $(\widetilde W_t)_{t\ge0}$.
\end{assumption}
Recall that the formal adjoint of $L$ can be written as
\begin{equation}
    {\bar L}\varphi=\ip{\bar F}{D\varphi}_\mH+\frac12\Tr(aD^2\varphi),\quad\forall\varphi\in\FC,
\end{equation}
so under \cref{asm:reversed-diffusion}, we may apply the same deduction as in \cref{prop:forward-generator-identification} to connect the operator ${\bar L}$ to the reversed transition semigroup.
\begin{proposition}\label{prop:reversed-generator-identification}
Under \cref{asm:stationary-diffusion,asm:cylinder-generator-integrability,asm:covariance-divergence,asm:reversed-diffusion}, the reversed transition semigroup $(\bar P_t)_{t\ge0}$ corresponding to \cref{eq:stationary-reversed-diffusion} extends to a strongly continuous contraction semigroup on $L^1(\mu)$.
The operator ${\bar L}$ acts the same as the restriction of the generator on $\FC$ up to a difference on $\mu$-null sets.
Moreover, for every $t\ge0$ and $\varphi\in\FC$,
\begin{equation}\label{eq:reversed-semigroup-integral}
    \bar P_t\varphi-\varphi =\int_0^t\bar P_s({\bar L}\varphi)\md s \qquad\text{in }L^1(\mu),
\end{equation}
where the time integral is a Bochner integral in $L^1(\mu)$.
\end{proposition}
In this section, we will explore when and how the EPR value can be explicitly evaluated.
To this end, we introduce the regular regime for our infinite-dimensional Hilbert-valued case as
\begin{equation}\label{eq:finite-cameron-martin-energy}
    \int_\mH g^2\md\mu<+\infty,\quad\text{where}\quad g(x)\coloneqq
    \begin{cases}
        \norm{a(x)^{-1/2}\Firr(x)}_\mH,&\Firr(x)\in\Ran\left(a(x)^{1/2}\right),\\
        +\infty,&\text{otherwise}.
    \end{cases}
\end{equation}
Here $a(x)^{-1/2}$ denotes the minimum-norm inverse of $a(x)^{1/2}$ on its range, with values in $(\ker a(x))^\perp$.
Note that the regular condition also implies that $\Firr(x)\in\Ran\left(a^{1/2}(x)\right)$ holds for $\mu$-a.e. $x$.
We first establish an extended lower bound, then obtain matching upper bounds in the regular regime by changing measure in either direction.

\subsection{A general lower bound for entropy production}\label{sec:entropy-lower-bound}
We begin with a pointwise dual representation of the extended Cameron--Martin energy $g^2$.
This representation, combined with discrete-time sampling and a variational bound for relative entropy, yields the EPR lower bound.
\begin{lemma}\label{lem:cameron-martin-duality}
For $x\in\mH$, define the concave quadratic $Q_x:\mH\to\R$ as
\begin{equation}
    Q_x(c)\coloneqq2\ip{\Firr(x)}{c}_\mH-\frac12\ip{c}{a(x)c}_\mH,\qquad c\in\mH .
\end{equation}
Then we have
\begin{equation}\label{eq:cameron-martin-duality}
    \sup_{c\in\mH}Q_x(c)=2g(x)^2\in[0,+\infty],
\end{equation}
and the mapping $x\mapsto Q_x(c)$ is measurable with respect to $x$.
\end{lemma}

\begin{proof}
For simplicity, we fix $x\in\mH$ and abbreviate $a=a(x)$ and $\Firr=\Firr(x)$.
For $\varepsilon>0$, we introduce
\begin{equation}\label{eq:regularized-energy-quadratic}
    Q_x^\varepsilon(c)\coloneqq2\ip{\Firr}{c}_\mH-\frac12\ip{c}{(a+\varepsilon I)c}_\mH=Q_x(c)-\frac\varepsilon2\norm{c}_\mH^2,\qquad\forall c\in\mH,
\end{equation}
then we have
\begin{equation}
    \sup_{c\in\mH} Q_x(c)=\sup_{c\in\mH}\sup_{\varepsilon>0}Q_x^\varepsilon(c)=\sup_{\varepsilon>0}\sup_{c\in\mH}Q_x^\varepsilon(c)
    =\sup_{\varepsilon>0}2\ip{\Firr}{(a+\varepsilon I)^{-1}\Firr}_\mH.
\end{equation}
By the spectral theorem and monotone convergence, we have the monotonic limit
\begin{equation}
    \lim_{\varepsilon\to0^+} \ip{\Firr}{(a+\varepsilon I)^{-1}\Firr}_\mH =
    \begin{cases}
        \norm{a^{-1/2}\Firr}_\mH^2, & \Firr\in\Ran(a^{1/2}),\\
        +\infty, & \text{otherwise},
    \end{cases}
\end{equation}
which proves \cref{eq:cameron-martin-duality}.
Since $\mH$ is separable and the mapping $x\mapsto Q_x(c)$ is Borel for each $c\in\mH$, the function $g^2$ is Borel measurable as a supremum over a countable dense subset of $\mH$.
\end{proof}
We next turn to a path-space lower bound, starting from the two-point law of the stationary forward process.
For each step size $\delta>0$, let
\begin{equation}
    p_\delta(x,\cdot)=\Law(X_\delta\mid X_0=x)
\end{equation}
be a Borel version of the forward transition kernel, and define
\begin{equation}\label{eq:stationary-two-point-law}
    \Pi_\delta(\md x,\md y)\coloneqq\mu(\md x)p_\delta(x,\md y)
\end{equation}
as the joint law of $(X_0,X_\delta)$.

Let $\tau\colon\mH\times\mH\to\mH\times\mH$ denote the transposition map.
By the definition of stationary time reversal, the corresponding two-point law of the reversed path measure is
\begin{equation}\label{eq:reversed-two-point-law}
    \bar\Pi_\delta=\tau_\#\Pi_\delta, \qquad \tau(x,y)=(y,x).
\end{equation}
Since both marginals of $\Pi_\delta$ are $\mu$ by stationarity, the first marginal of $\bar\Pi_\delta$ is also $\mu$.
Hence, by disintegration on the Polish space $\mH$, there exists a Borel transition kernel $\bar p_\delta$ such that
\begin{equation}
    \bar\Pi_\delta(\md x,\md y) =\mu(\md x)\bar p_\delta(x,\md y).
\end{equation}
Equivalently, $\bar p_\delta$ is the transition kernel of the stationary time-reversed Markov process.

We define the one-step relative entropy by
\begin{equation}\label{eq:one-step-relative-entropy}
    I(\delta)\coloneqq\int_\mH\KLdiv{p_\delta(x,\cdot)}{\bar p_\delta(x,\cdot)}\mu(\md x).
\end{equation}
Since $\Pi_\delta$ and $\bar\Pi_\delta$ have the same first marginal $\mu$, the chain rule for relative entropy yields
\begin{equation}\label{eq:two-point-entropy-chain-rule}
    \KLdiv{\Pi_\delta}{\bar\Pi_\delta} =\KLdiv{\mu}{\mu}+\int_\mH\KLdiv{p_\delta(x,\cdot)}{\bar p_\delta(x,\cdot)}\mu(\md x) =I(\delta)
\end{equation}
in the extended interval $[0,+\infty]$.

\begin{lemma}\label{lem:two-point-short-time-expansion}
Under \cref{asm:stationary-diffusion,asm:cylinder-generator-integrability,asm:covariance-divergence}, for every smooth bounded cylinder function $\Psi$ on $\mH\times\mH$ of the form
\begin{equation}
    \Psi(x,y)=f\left(\ip{x}{p_1}_\mH,\cdots,\ip{x}{p_m}_\mH,\ip{y}{q_1}_\mH,\cdots,\ip{y}{q_n}_\mH\right),\qquad f\in C_b^\infty(\mR^{m+n}),
    \quad\{p_i\}_{i=1}^m\cup\{q_j\}_{j=1}^n\subseteq\mH,
\end{equation}
we have
\begin{equation}\label{eq:two-point-short-time-expansion}
    \E[\Psi(X_0,X_\delta)]=\int_\mH\Psi(x,x)\mu(\md x)+\delta\int_\mH(L_2\Psi)(x,x)\mu(\md x)+o(\delta),\qquad0<\delta\ll1.
\end{equation}
\end{lemma}
\begin{proof}
Fixing the first variable as $X_0$ and applying finite-dimensional It\^o's formula to the second variable, we obtain
\begin{equation}\label{eq:two-point-ito-formula}
    \Psi(X_0,X_\delta)-\Psi(X_0,X_0)=\int_0^\delta(L_2\Psi)(X_0,X_s)\md s+\int_0^\delta\ip{B(X_s)^*D_2\Psi(X_0,X_s)}{\md W_s}_\mU.
\end{equation}
Here we define
\begin{equation}\label{eq:second-variable-generator}
    L_2\Psi(x,y)=\ip{F(y)}{D_2\Psi(x,y)}_\mH+\frac12\Tr(a(y)D_{22}^2\Psi(x,y))
\end{equation}
as the operator $L$ acting on the second variable.
The differential operators $D_2$ and $D_{22}^2$ are the Fr\'echet gradient and Hessian of $\Psi$ with respect to the second Hilbert-space variable, respectively.
Since the martingale term in \cref{eq:two-point-ito-formula} is square integrable and thus has zero expectation by \cref{lem:generator-integrability-characterization}, it follows by taking expectation that
\begin{equation}\label{eq:two-point-expected-increment}
    \E[\Psi(X_0,X_\delta)-\Psi(X_0,X_0)] =\int_0^\delta\E[(L_2\Psi)(X_0,X_s)]\md s.
\end{equation}
By directly expanding $(L_2\Psi)(X_0,X_s)$ according to \cref{eq:second-variable-generator}, we can write it as a finite sum of terms $c(X_s)\alpha(X_0,X_s)$, where $c\in L^1(\mu)$ by \cref{lem:generator-integrability-characterization} and $\alpha$ is bounded and continuous.
For each summand, we will show that
\begin{equation}\label{eq:stationary-weighted-continuity}
    \E c(X_s)\alpha(X_0,X_s)-\E c(X_s)\alpha(X_s,X_s)\to0,\qquad s\to0^+.
\end{equation}
When $c$ is bounded, the limit is trivial by bounded convergence.
For a general $c\in L^1(\mu)$, we employ the truncation argument with a positive integer $N$ as follows.
For $|c(X_s)|\le N$, we have
\begin{equation}
    \E\left[|c(X_s)\alpha(X_0,X_s)-c(X_s)\alpha(X_s,X_s)|\cdot\mathbb 1_{|c(X_s)|\le N}\right]\le N\E|\alpha(X_0,X_s)-\alpha(X_s,X_s)|\to0
\end{equation}
as $s\to0^+$ by bounded convergence.
Otherwise
\begin{equation}
    \E\left[|c(X_s)\alpha(X_0,X_s)-c(X_s)\alpha(X_s,X_s)|\cdot\mathbb 1_{|c(X_s)|> N}\right]
    \le \sup_{s\ge0}\E\left[|c(X_s)|\cdot\mathbb 1_{|c(X_s)|>N}\right]\cdot2\norm{\alpha}_\infty,
\end{equation}
where the supremum does not depend on $s$ by stationarity, and the upper bound converges to 0 as $N\to+\infty$.
Therefore, we have \cref{eq:stationary-weighted-continuity}, which implies
\begin{equation}\label{eq:two-point-generator-continuity}
    \E (L_2\Psi)(X_0,X_s)-\E (L_2\Psi)(X_s,X_s)\to0,\qquad s\to0^+.
\end{equation}
Consequently, \cref{eq:two-point-expected-increment,eq:two-point-generator-continuity}, together with stationarity, yield \cref{eq:two-point-short-time-expansion}.
\end{proof}
\begin{theorem}\label{thm:entropy-production-lower-bound}
Under \cref{asm:stationary-diffusion,asm:cylinder-generator-integrability,asm:covariance-divergence}, we have the following lower bound
\begin{equation}\label{eq:entropy-production-lower-bound}
    \epr\ge2\int_\mH g^2\md\mu\in[0,+\infty].
\end{equation}
Hence $\epr=+\infty$ whenever \cref{eq:finite-cameron-martin-energy} fails.
\end{theorem}
\begin{proof}
For any $\delta>0$, we consider the mapping
\begin{equation}
    \pi_\delta\colon C([0,\delta];\mH)\to\mH\times\mH,\qquad \omega\mapsto(\omega_0,\omega_\delta),
\end{equation}
then by \cref{prop:entropy-time-linearity},
\begin{equation}\label{eq:two-point-data-processing}
    \epr=\frac1\delta\KLdiv{\Prob_{[0,\delta]}}{\Prob^{\mathrm{rev}}_{[0,\delta]}}
    \ge\frac1\delta\KLdiv{(\pi_\delta)_\#\Prob_{[0,\delta]}}{(\pi_\delta)_\#\Prob^{\mathrm{rev}}_{[0,\delta]}}=\frac{I(\delta)}{\delta},
\end{equation}
where the inequality is data processing for $\pi_\delta$.

To obtain a lower bound for $I(\delta)$, we make use of the variational inequality for relative entropy.
It can be shown that for every bounded Borel $\beta\colon\mH\times\mH\to\mR$,
\begin{equation}\label{eq:entropy-variational-bound}
    I(\delta)\ge\int \beta\md \Pi_\delta-\log\int \me^\beta\md \bar\Pi_\delta .
\end{equation}
Indeed, by the boundedness of $\beta$, we define the exponential tilt $\md \bar\Pi_\delta^\beta\coloneqq Z^{-1}\me^\beta\md \bar\Pi_\delta$ with the normalizer $Z\coloneqq\int \me^\beta\md \bar\Pi_\delta$.
If $I(\delta)=+\infty$, then \cref{eq:entropy-variational-bound} is immediate.
Otherwise, we have
\begin{equation}
    0\le\KLdiv{\Pi_\delta}{\bar\Pi_\delta^\beta} =\KLdiv{\Pi_\delta}{\bar\Pi_\delta}-\int \beta\md \Pi_\delta+\log Z,
\end{equation}
which rearranges to \cref{eq:entropy-variational-bound}.

It suffices to find an appropriate function $\beta$ to obtain estimations for both the two terms on the right-hand side of \cref{eq:entropy-variational-bound}.
Consider a smooth bounded cylinder function $\beta$ on $\mH\times\mH$ of the following form
\begin{equation}\label{eq:antisymmetric-two-point-test}
    \beta(x,y)\coloneqq\frac12\sum_{j=1}^n(f_j(x)+f_j(y))\sin\ip{h_j}{y-x}_\mH,\qquad\{h_j\}_{j=1}^n\subseteq\mH,\quad\{f_j\}_{j=1}^n\subseteq\FC.
\end{equation}
By direct calculation, we have $\beta(y,x)=-\beta(x,y)$, $\beta(x,x)=0$, and
\begin{equation}\label{eq:antisymmetric-test-diagonal-derivatives}
    \begin{aligned}
        D_2\beta(x,x)&=\sum_{j=1}^n f_j(x)h_j,\\
        (L_2\beta)(x,x)&=\sum_{j=1}^n\left(f_j(x)\ip{F(x)}{h_j}_\mH+\frac12\ip{a(x)h_j}{Df_j(x)}_\mH\right).
    \end{aligned}
\end{equation}
Applying \cref{lem:two-point-short-time-expansion} to $\beta$ yields
\begin{equation}\label{eq:antisymmetric-test-expectation}
    \int\beta\md\Pi_\delta=\delta\int_\mH (L_2\beta)(x,x)\mu(\md x) +o(\delta).
\end{equation}
By \cref{eq:reversed-two-point-law} and $\beta(x,x)=0$,
\begin{equation}
    \begin{aligned}
        \int\me^\beta\md\bar\Pi_\delta=\int\me^{-\beta}\md\Pi_\delta&=\int_\mH\me^{-\beta(x,x)}\mu(\md x)+\delta\int_\mH(L_2\me^{-\beta})(x,x)\mu(\md x)+o(\delta)\\
        &=1+\delta\int_\mH\left(-L_2\beta+\frac12\ip{aD_2\beta}{D_2\beta}_\mH\right)(x,x)\mu(\md x)+o(\delta).
    \end{aligned}
\end{equation}
Therefore, combining these expansions with \cref{eq:two-point-data-processing,eq:entropy-variational-bound} and taking $\delta\to0^+$, we obtain
\begin{equation}\label{eq:entropy-lower-bound-test-field}
    \epr\ge\int_\mH\left(2\ip{F_\irr}{V}_\mH-\frac12\ip{aV}{V}_\mH\right)\md\mu=\int_\mH Q_x(V(x))\mu(\md x),
\end{equation}
where $V(x)\coloneqq D_2\beta(x,x)$ and the term $\ip{F_\irr}{V}_\mH$ appears by the integration-by-parts identity
\begin{equation}
    \int_\mH (L_2\beta)(x,x)\mu(\md x) =\int_\mH\sum_{j=1}^n f_j\ip{F-\frac12d}{h_j}_\mH\md\mu =\int_\mH\ip{\Firr}{V}_\mH\md\mu
\end{equation}
according to \cref{asm:covariance-divergence}.

It remains to optimize over $V$.
Choose a countable dense set $\{h_k\}_{k=0}^\infty\subseteq\mH$ with $h_0=0$, and define
\begin{equation}
    q_n(x)\coloneqq\max_{0\le k\le n}Q_x(h_k),\qquad k_n(x)\coloneqq\min\{0\le k\le n:Q_x(h_k)=q_n(x)\}.
\end{equation}
Clearly, the sets $H_k\coloneqq\{k_n=k\}$ form a Borel partition, and $q_n\ge0$.
For each $k$, choose $f_{k,l}\in\FC$ such that $0\le f_{k,l}\le1$ and $f_{k,l}\to\mathbf1_{H_k}$ $\mu$-almost everywhere as $l\to+\infty$.
By setting $V_l^{(n)}\coloneqq\sum_{k=0}^n f_{k,l}h_k$, we have $V_l^{(n)}\to h_{k_n}$ $\mu$-almost everywhere, and
\begin{equation}
    \left|Q_x(V_l^{(n)}(x))\right| \le 2\sum_{k=0}^n\left|\ip{\Firr(x)}{h_k}_\mH\right| +\frac{n+1}2\sum_{k=0}^n\ip{h_k}{a(x)h_k}_\mH.
\end{equation}
The right-hand side is integrable by \cref{lem:generator-integrability-characterization,asm:covariance-divergence}.
Dominated convergence in \cref{eq:entropy-lower-bound-test-field} therefore yields $\epr\ge\int_\mH q_n\md\mu$ for every $n$.
Combined with density and continuity of $c\mapsto Q_x(c)$, monotone convergence and \cref{lem:cameron-martin-duality} now lead to
\begin{equation}
    \epr \ge\lim_{n\to\infty}\int_\mH q_n\md\mu =\int_\mH\sup_{c\in\mH}Q_x(c)\mu(\md x) =2\int_\mH g^2\md\mu,
\end{equation}
which completes our proof.
\end{proof}

The lower bound reveals two distinct mechanisms for infinite entropy production.
The irreversible drift may fail to belong to the local Cameron--Martin space on a set of positive $\mu$-measure, or it may belong to that space almost everywhere while its Cameron--Martin energy is not integrable.
Either mechanism forces $\epr=+\infty$.

\subsection{Sharpness in the regular regime}\label{sec:finite-energy-identity}

We now turn to the regular regime with the finite-energy condition \cref{eq:finite-cameron-martin-energy}.
We first identify the minimal shift in the noise space that realizes the drift difference $2\Firr$.
Girsanov's theorem is employed to obtain a matching upper bound, showing that the lower bound in \cref{thm:entropy-production-lower-bound} is sharp.
\begin{lemma}\label{lem:measurable-minimal-noise-control}
We define a Borel map $u\colon\mH\to \mU$ such that
\begin{equation}
    u(x)=
    \begin{cases}
        2B(x)^+\Firr(x),&x\in R,\\
        0,&\text{otherwise},
    \end{cases}
    \qquad R\coloneqq\left\{x\in\mH\mid \Firr(x)\in\Ran(a(x)^{1/2})\right\},
\end{equation}
where $B(x)^+$ denotes the Moore--Penrose inverse on its range.
Moreover, for $x\in R$ we have
\begin{equation}\label{eq:minimal-control-energy}
    B(x)u(x)=2\Firr(x),\qquad \norm{u(x)}_{\mU}^2 =4\norm{a(x)^{-1/2}\Firr(x)}_\mH^2 =4g(x)^2.
\end{equation}
\end{lemma}

\begin{proof}
By \cref{lem:cameron-martin-duality}, the set $R=\{g^2<+\infty\}$ is Borel.
For $\varepsilon>0$, define
\begin{equation}
    u_\varepsilon(x)\coloneqq2B(x)^*\left(a(x)+\varepsilon I\right)^{-1}\Firr(x),
\end{equation}
which is Borel by the measurability of the coefficients and the continuity of the resolvent.

Fix $x\in R$ and suppress it from the notation.
Since $\Ran B=\Ran(a^{1/2})$, the vector $\widetilde u\coloneqq2B^+\Firr$ satisfies $B\tilde u=2\Firr$ and $\tilde u\in(\ker B)^\perp$.
The resolvent identity and spectral calculus give
\begin{equation}\label{eq:minimal-control-resolvent-limit}
    u_\varepsilon=B^*(BB^*+\varepsilon I)^{-1}B\widetilde u=B^*B(B^*B+\varepsilon I)^{-1}\widetilde u\to\tilde u\quad\text{in}\quad\mU\quad\text{as}\quad\varepsilon\to0^+.
\end{equation}
Therefore, the pointwise limit
\begin{equation}\label{eq:measurable-control-limit}
    u\coloneqq\lim_{n\to\infty}\mathbf1_Ru_{1/n}
\end{equation}
is Borel with the asserted values and satisfies $B(x)u(x)=2\Firr(x)$ for every $x\in R$.

It remains to identify the norm.
We still fix $x\in R$ and suppress it from the notation.
Since $Bu=2\Firr$ and $u\in(\ker B)^\perp=\overline{\Ran B^*}$, completing the square and applying \cref{lem:cameron-martin-duality} yield
\begin{equation}
    2g^2=\sup_{c\in\mH}\left(\ip{u}{B^*c}_\mU-\frac12\norm{B^*c}_\mU^2\right)=\frac12\norm{u}_\mU^2-\frac12\inf_{c\in\mH}\norm{u-B^*c}_\mU^2=\frac12\norm{u}_\mU^2,
\end{equation}
which implies \cref{eq:minimal-control-energy}.
\end{proof}

\begin{lemma}\label{lem:girsanov-entropy-identity}
Assume \cref{asm:stationary-diffusion,asm:cylinder-generator-integrability,asm:covariance-divergence,asm:reversed-diffusion,eq:finite-cameron-martin-energy} and fix $T>0$.
We define
\begin{equation}\label{eq:reversal-stochastic-exponential}
    \begin{aligned}
        M_t&\coloneqq\int_0^t\ip{u(X_s)}{\md W_s}_\mU,\\
        Z_t&\coloneqq\exp\left(-M_t-\frac12[M]_t\right)=\exp\left(-\int_0^t\ip{u(X_s)}{\md W_s}_\mU-\frac12\int_0^t\norm{u(X_s)}_\mU^2\md s\right)
    \end{aligned}
\end{equation}
for every $t\in[0,T]$, where $[M]_t$ stands for the quadratic variation.
If $(Z_t)_{0\le t\le T}$ is a true $\Prob$-martingale and the weak solutions of \cref{eq:stationary-reversed-diffusion} with initial law $\mu$ are unique in law, then
\begin{equation}\label{eq:finite-energy-entropy-identity}
    \Prob_{[0,T]} \sim \Prob^{\mathrm{rev}}_{[0,T]}, \qquad \epr =\frac1T\KLdiv{\Prob_{[0,T]}}{\Prob^{\mathrm{rev}}_{[0,T]}} =2\int_\mH g^2\md\mu <+\infty.
\end{equation}
\end{lemma}

\begin{proof}
By stationarity and \cref{eq:minimal-control-energy},
\begin{equation}\label{eq:stationary-control-energy}
    \E_\Prob\left[ \int_0^T\norm{u(X_s)}_\mU^2\md s \right] =T\int_\mH\norm{u}_\mU^2\md\mu =4T\int_\mH g^2\md\mu <+\infty.
\end{equation}
Hence $M$ is a square-integrable martingale and $\E_\Prob[M_T]=0$.
Moreover, $\mu(R)=1$ by \cref{eq:finite-cameron-martin-energy}, so stationarity and Fubini's theorem give $B(X_s)u(X_s)=2\Firr(X_s)$ for $\md s\otimes\Prob$-almost every $(s,\omega)$.

Define $\md\widehat\Prob_T=Z_T\md\Prob$ on $\mathcal F_T$.
By \cref{thm:girsanov-change-of-measure}, applied with $\psi_s=-u(X_s)$, we have $\widehat\Prob_T\sim\Prob$, and under $\widehat\Prob_T$ the process $X$ solves \cref{eq:stationary-reversed-diffusion} with initial law $\mu$.
Uniqueness in law therefore yields
\begin{equation}\label{eq:tilted-reversed-path-law}
    \widehat\Prob_T\circ\bigl(X|_{[0,T]}\bigr)^{-1} =\Prob^{\mathrm{rev}}_{[0,T]}\sim\Prob_{[0,T]}.
\end{equation}

The path-space density formula in \cref{thm:girsanov-change-of-measure} and conditional Jensen give
\begin{equation}\label{eq:entropy-upper-bound}
    \begin{aligned}
        \KLdiv{\Prob_{[0,T]}}{\Prob^{\mathrm{rev}}_{[0,T]}} &= \E_\Prob\left[ -\log\E_\Prob[Z_T\mid\mathcal F_T^X] \right]\\
        &\le \E_\Prob[-\log Z_T] =\frac12\E_\Prob\left[ \int_0^T\norm{u(X_s)}_\mU^2\md s \right]=2T\int_\mH g^2\md\mu.
    \end{aligned}
\end{equation}
Combining this with \cref{prop:entropy-time-linearity,thm:entropy-production-lower-bound} proves \cref{eq:finite-energy-entropy-identity}.
\end{proof}
Under the assumptions of \cref{lem:girsanov-entropy-identity}, equality holds in the conditional Jensen inequality in \cref{eq:entropy-upper-bound}.
Strict convexity of $r\mapsto-\log r$ and the density formula in \cref{thm:girsanov-change-of-measure} therefore give
\begin{equation}\label{eq:state-path-reversal-density}
    Z_T =\E_\Prob\left[Z_T\mid\mathcal F_T^X\right] = \frac{\md\Prob^{\mathrm{rev}}_{[0,T]}} {\md\Prob_{[0,T]}} \bigl(X|_{[0,T]}\bigr) \qquad\Prob\text{-a.s.}
\end{equation}
Thus $Z_T$ admits a state-path-measurable version as a consequence of the entropy identity, rather than an additional assumption.

Now we give a sufficient condition of Lipschitz continuity for the true-martingale condition required in \cref{lem:girsanov-entropy-identity}.

\begin{proposition}\label{prop:reversed-lipschitz-martingale}
Assume \cref{asm:stationary-diffusion,asm:cylinder-generator-integrability,asm:covariance-divergence,asm:reversed-diffusion} and the regular condition \cref{eq:finite-cameron-martin-energy}.
If $\bar F:\mH\to\mH$ and $B:\mH\to\HS(\mU,\mH)$ are Lipschitz on bounded sets, then the process $(Z_t)_{0\le t\le T}$ defined in \cref{eq:reversal-stochastic-exponential} is a true $\Prob$-martingale for every $T>0$.
Consequently, \cref{eq:finite-energy-entropy-identity} holds since the assumed Lipschitz continuity of $\bar F$ and $B$, together with \cref{asm:reversed-diffusion}, imply uniqueness in law from $\mu$.
\end{proposition}
\begin{proof}
For fixed $T>0$, we define
\begin{equation}
    A_t\coloneqq\int_0^t\norm{u(X_s)}_\mU^2\md s,\qquad\tau_n\coloneqq\inf\{t>0\mid A_t\ge n\},\quad\forall n\in\mN,
\end{equation}
then by \cref{eq:stationary-control-energy}, $A_T<+\infty$ $\Prob$-almost surely.
Since $A_{T\wedge\tau_n}\le n$, Novikov's criterion shows that $(Z_{t\wedge\tau_n})_{0\le t\le T}$ is a true martingale.
For each $n\in\mN$, we define a measure $\mQ_n$ such that $\md\mQ_n\coloneqq Z_{T\wedge\tau_n}\md\mP$ on $\mathcal F_T$.
By \cref{thm:girsanov-change-of-measure}, $\mQ_n\sim\mP$ on $\mathcal F_T$ with the same initial distribution $\mu$ and the process
\begin{equation}
    W_t^{(n)}\coloneqq W_t+\int_0^{t\wedge \tau_n}u(X_s)\md s
\end{equation}
is a cylindrical Wiener process under $\mQ_n$.
Consequently, $X$ solves the reversed equation
\begin{equation}
    \md X_t=\bar F(X_t)\md t+B(X_t)\md W_t^{(n)}
\end{equation}
up to $T\wedge\tau_n$ by $\Prob$-stationarity and \cref{eq:minimal-control-energy}.

The Lipschitz bounds imply pathwise uniqueness up to the first time either solution leaves a ball.
Since continuous paths are bounded on $[0,T]$, letting the radius tend to infinity proves pathwise uniqueness for global solutions.
By \cref{asm:reversed-diffusion}, the reversed equation already has a global weak solution with initial law $\mu$.
Hence \citet[Theorem~2]{Ondrejat2004} implies strong existence and uniqueness in law for this initial distribution.
Let $Y^{(n)}$ be the corresponding strong solution on $[0,T]$ under $\mQ_n$, driven by $W^{(n)}$ and with $Y_0^{(n)}=X_0$.
The same stopped uniqueness argument shows that $X_t=Y_t^{(n)}$ for all $0\le t\le T\wedge\tau_n$, $\mQ_n$-almost surely.
By uniqueness in law and \cref{asm:reversed-diffusion}, $Y^{(n)}|_{[0,T]}$ has law $\mP_{[0,T]}^\rev$ under $\mQ_n$.
Consequently, we have
\begin{equation}
    \begin{aligned}
        \mQ_n\left\{\tau_n>T\right\}&=\mP_{[0,T]}^\rev\left\{\int_0^T\norm{u(\omega_s)}_\mU^2\md s<n\right\}\\
        &\ge1-\frac1n\E_{\mP_{[0,T]}^\rev}\int_0^T\norm{u(\omega_s)}_\mU^2\md s=1-\frac{4T}{n}\int_\mH g^2\md\mu\to1
    \end{aligned}
\end{equation}
as $n\to\infty$.
Therefore,
\begin{equation}
    \E_{\mP}[Z_T]=\lim_{n\to\infty}\E_{\mP}\left[Z_T\mathbf 1_{\{\tau_n>T\}}\right]=\lim_{n\to\infty}\E_{\mP}\left[Z_{T\wedge\tau_n}\mathbf 1_{\{\tau_n>T\}}\right]
    =\lim_{n\to\infty}\mQ_n\{\tau_n>T\}=1
\end{equation}
by the definition of $\mQ_n$, which completes our proof.
\end{proof}

Alternatively, if we assume the Lipschitz continuity of the forward drift instead of the reversed drift, then we can directly obtain the explicit formula of the EPR value when the regular condition \cref{eq:finite-cameron-martin-energy} holds.

\begin{proposition}\label{prop:forward-lipschitz-entropy-identity}
Assume \cref{asm:stationary-diffusion,asm:cylinder-generator-integrability,asm:covariance-divergence,asm:reversed-diffusion} and the regular condition \cref{eq:finite-cameron-martin-energy}.
If $F$ and $B$ are Lipschitz on bounded sets, then \cref{eq:finite-energy-entropy-identity} holds.
\end{proposition}

\begin{proof}
Fix $T>0$.
As in the proof of \cref{prop:reversed-lipschitz-martingale}, the Lipschitz continuity of $F,B$, together with the global weak solution supplied by \cref{asm:stationary-diffusion}, yields strong existence and uniqueness in law for the forward equation with initial distribution $\mu$.

Work on a probability space carrying the stationary reversed solution $\bar X$ from \cref{asm:reversed-diffusion}, with cylindrical Wiener process $\bar W$ and probability measure $\bar\Prob$.
Define
\begin{equation}
    \bar A_t=\int_0^t\norm{u(\bar X_s)}_\mU^2\,\md s, \qquad \bar Z_t= \exp\left( \int_0^t\ip{u(\bar X_s)}{\md\bar W_s}_\mU -\frac12\bar A_t \right).
\end{equation}
By stationarity and \cref{eq:stationary-control-energy},
\begin{equation}
    \E_{\bar\Prob}\bar A_T=\E_{\Prob_{[0,T]}} \int_0^T\norm{u(\omega_s)}_\mU^2\,\md s=4T\int_\mH g^2\md\mu<+\infty.
\end{equation}
The localized Girsanov argument in the proof of \cref{prop:reversed-lipschitz-martingale} therefore applies with the forward and reversed equations interchanged and $u$ replaced by $-u$: the tilted drift is $\bar F+Bu=F$, and the control energy is integrable under the target stationary path law $\Prob_{[0,T]}$.
Consequently, $(\bar Z_t)_{0\le t\le T}$ is a true $\bar\Prob$-martingale.

Define $\mQ$ on $\mathcal F_T$ by $\md\mQ=\bar Z_T\md\bar\Prob$.
Under $\mQ$, the process
\begin{equation}
    W_t^{\mQ}\coloneqq\bar W_t-\int_0^t u(\bar X_s)\md s
\end{equation}
is a cylindrical Wiener process, and $\bar X$ solves the forward equation
\begin{equation}
    \md\bar X_t=F(\bar X_t)\md t+B(\bar X_t)\md W_t^\mQ
\end{equation}
with initial law $\mu$.
By uniqueness in law, its path law under $\mQ$ is $\Prob_{[0,T]}$.
Since $\bar Z_T>0$, the measures $\mQ$ and $\bar\Prob$ are equivalent on $\mathcal F_T$, and hence their induced path laws are equivalent.
Moreover, we have
\begin{equation}
    \log\bar Z_T=\int_0^T\ip{u(\bar X_s)}{\md W_s^{\mQ}}_\mU+\frac12\bar A_T.
\end{equation}
Since $\E_{\mQ}\bar A_T<+\infty$ by stationarity, the stochastic integral is a square-integrable $\mQ$-martingale, and thus
\begin{equation}
    \E_\mQ\left[\log\bar Z_T\right]=\frac12\E_\mQ\left[\bar A_T\right]=2T\int_\mH g^2\md\mu.
\end{equation}
Consequently, the data-processing inequality implies
\begin{equation}
    \KLdiv{\Prob_{[0,T]}}{\Prob^{\mathrm{rev}}_{[0,T]}}\le\KLdiv{\mQ}{\bar\Prob}= \E_{\mQ}\left[\log\bar Z_T\right]=2T\int_\mH g^2\md\mu.
\end{equation}
The reverse inequality follows from \cref{thm:entropy-production-lower-bound} and \cref{prop:entropy-time-linearity}, completing the proof.
\end{proof}
Combining \cref{thm:entropy-production-lower-bound,lem:girsanov-entropy-identity,prop:reversed-lipschitz-martingale,prop:forward-lipschitz-entropy-identity} yields the main theorem of this paper.

\begin{theorem}\label{thm:entropy-production-dichotomy}
Assume \cref{asm:stationary-diffusion,asm:cylinder-generator-integrability,asm:covariance-divergence}.
\begin{itemize}
    \item If \cref{eq:finite-cameron-martin-energy} fails, then $\epr=+\infty$.
    \item If \cref{eq:finite-cameron-martin-energy,asm:reversed-diffusion} hold, and if the diffusion coefficient $B$ and at least one of the drifts $F$ and $\bar F$ are Lipschitz on bounded sets, then for every $T>0$,
    \begin{equation}
        \Prob_{[0,T]}\sim\Prob^{\mathrm{rev}}_{[0,T]}\quad\text{and}\quad\epr=2\int_\mH g^2\md\mu<+\infty.
    \end{equation}
\end{itemize}
\end{theorem}

\section{Time reversal}\label{sec:time-reversal}
This section gives two criteria for the diffusion representation in \cref{asm:reversed-diffusion}.
The formal identities $\bar F=F-2\Firr=d-F$ and $\bar L=L_{\mathrm{rev}}-L_{\mathrm{irr}}$ established in \cref{prop:reversible-irreversible-decomposition} have provided candidates for the drift and the generator, and the remaining task is to show that they in fact correspond to the real stationary reversal.

Both criteria are based on the abstract integration-by-parts result of \citet{CCGL2023}.
\Cref{sec:constant-diffusion-time-reversal} treats constant diffusion coefficients and isolates the additional path integrability needed to recover an $\mH$-valued equation.
\Cref{sec:state-dependent-time-reversal} works directly on $\mH$ for continuous state-dependent diffusion coefficients, under finite Cameron--Martin energy and additional norm-integrability assumptions.
The entropy conclusions then follow under the corresponding conditions of \cref{thm:entropy-production-dichotomy}.

\subsection{Constant diffusion coefficient}\label{sec:constant-diffusion-time-reversal}
For a constant diffusion coefficient, the integration-by-parts structure in \cref{asm:covariance-divergence} identifies the reversed coordinate dynamics.
The following result also characterizes the additional path integrability needed to recover an $\mH$-valued equation.
Its proof uses the abstract criterion of \citet[Theorem~3.6]{CCGL2023}, without a generator-core, sector, or path-space entropy assumption.

Throughout this subsection, we assume that $B(x)\equiv B$ is constant.
Choose a complete orthonormal eigenbasis $(e_k)_{k\ge0}$ of $\mH$ such that
\begin{equation}\label{eq:constant-covariance-spectrum}
    ae_k=q_ke_k, \qquad q_k\geq0, \qquad \Tr a=\sum_{k\ge0}q_k<+\infty.
\end{equation}
Let $\bar X$ denote the canonical process with the consistent stationary reversed path laws.

\begin{theorem}\label{thm:constant-diffusion-time-reversal}
Suppose that \cref{asm:stationary-diffusion,asm:cylinder-generator-integrability,asm:covariance-divergence} hold and that $B(x)\equiv B$ is constant.
For every $k\in\mN$, the stationary reversal satisfies
\begin{equation}\label{eq:constant-reversed-coordinates}
    \ip{\bar X_t}{e_k}_{\mH} =\ip{\bar X_0}{e_k}_{\mH} +\int_0^t\ip{\bar F(\bar X_s)}{e_k}_{\mH}\md s+M_t^k,
\end{equation}
where the processes $M^k$ are continuous square-integrable martingales in the usual augmentation of the reversed natural filtration, starting from zero, with covariations
\begin{equation}\label{eq:constant-coordinate-covariations}
    \left[M^i,M^j\right]_t=q_i\delta_{ij}t.
\end{equation}
For each $T>0$, this stationary reversal admits an $\mH$-valued weak realization
\begin{equation}\label{eq:constant-reversed-diffusion}
    \md\bar X_t=\bar F(\bar X_t)\md t+B(\bar X_t)\md\widetilde W_t, \qquad \bar X_0\sim\mu, \qquad 0\leq t\leq T,
\end{equation}
with a Bochner-integrable drift if and only if
\begin{equation}\label{eq:reversed-drift-path-integrability}
    \int_0^T\norm{d(X_s)}_{\mH}\md s<+\infty,\quad\text{or equivalently,}\quad\int_0^T\norm{\Firr(X_s)}_{\mH}\md s<+\infty\qquad\Prob\text{-a.s.}
\end{equation}
Here $\widetilde W$ is a cylindrical Wiener process on $\mU$, on an extension of the stochastic basis if necessary.

In particular, if \cref{eq:finite-cameron-martin-energy} holds, then \cref{eq:reversed-drift-path-integrability} holds for every $T>0$, which yields \cref{asm:reversed-diffusion}.
\end{theorem}

\begin{proof}
Fix $T>0$.
By \cref{lem:generator-integrability-characterization,asm:cylinder-generator-integrability,asm:covariance-divergence}, we have $\bar F_h\coloneqq\ip{d-F}{h}_\mH\in L^1(\mu)$.
Define
\begin{equation}\label{eq:coordinate-embedding}
    \mJ\colon\mH\to\mR^{\mN},\qquad\mJ(x)=\left(\ip{x}{e_k}_{\mH}\right)_{k\ge0};\qquad Y=\mJ X,\qquad \nu=\mJ_\#\mu.
\end{equation}
With the product topology on $\mR^{\mN}$, the map $\mJ$ is a Borel isomorphism onto the Borel subset $\ell^2\subseteq\mR^\mN$, which implies that $Y$ is a stationary Markov process with continuous paths.

Consider the function class $\mathcal A$ as the algebra of smooth bounded finite-coordinate functions on $\mR^{\mN}$ with all derivatives bounded.
A key observation is that the algebra \(\mathcal A\) determines weak convergence of probability measures on \(\mathbb R^{\mathbb N}\) since convergence against $\mathcal A$ gives convergence of all finite-dimensional marginals, and tightness of the coordinate marginals gives tightness on the product space.

For $f,g\in\mathcal A$, put $u=f\circ\mJ$ and $v=g\circ\mJ$.
By It\^o's formula, we have a forward extended generator satisfying $(\mathcal L_Yf)\circ\mJ=Lu$ and a continuous version of its carr\'e du champ,
\begin{equation}\label{eq:coordinate-carre-du-champ}
    \Gamma_Y(f,g)(y)=\sum_{k\ge0}q_k\partial_kf(y)\partial_kg(y),\qquad\forall y\in\ell^2.
\end{equation}
By the definition of $\mathcal A$, the sum is finite, so $\Gamma_Y(f,g)$ is bounded and continuous on $\mR^{\mN}$.
Moreover, $\mathcal L_Yf\in L^1(\nu)$, and the martingale part in It\^o's formula is square integrable.

The spatial integration-by-parts identity gives
\begin{equation}\label{eq:coordinate-spatial-integration-by-parts}
    \int_{\mR^{\mN}}\Gamma_Y(f,g)\md\nu= \int_{\mH}\ip{aDu}{Dv}_\mH\md\mu=-\int_{\mH}\left(\divg_\mu(aDu)\right)v\md\mu=-\int_{\mH}S_uv\md\mu,
\end{equation}
where
\begin{equation}\label{eq:coordinate-divergence-density}
    S_u:=\divg_\mu(aDu) =\Tr(aD^2u)+\ip{d}{Du}_{\mH}\in L^1(\mu),\qquad\forall u\in\FC.
\end{equation}

To apply the integration-by-parts formula given by \citet{CCGL2023}, we set the configuration space as $\mR^\mN$ and introduce the normalized class
\begin{equation}
    \mathcal A_0\coloneqq\left\{f\in\mathcal A\mid \norm{f}_\infty+ \sum_{k\ge0}2^{-k}\norm{\partial_kf}_\infty\leq1\right\}.
\end{equation}
This class is closed under products and still determines weak convergence, since every element of $\mathcal A$ is a scalar multiple of an element of $\mathcal A_0$.
For fixed $f\in\mathcal A_0$ and $w\in C_b([0,T]\times\mR^{\mN})$ with $w_t\in\mathcal A_0$,
\begin{equation}
    \sup_{t,y}|\Gamma_Y(f,w_t)(y)|\leq\sum_{k\ge0}2^kq_k\norm{\partial_kf}_\infty<\infty,
\end{equation}
where the finiteness follows from the fact that $f$ depends on finitely many coordinates.
Hence the spacetime integrals below are absolutely integrable.
We consider the finite signed measure
\begin{equation}\label{eq:coordinate-spacetime-measure}
    \sigma_f(\md t,\md y)\coloneqq-S_u\left(\mJ^{-1}(y)\right)\md t\,\nu(\md y),
\end{equation}
where the density is understood $\nu$-almost everywhere.
By Fubini's theorem and \cref{eq:coordinate-spatial-integration-by-parts}, we have
\begin{equation}\label{eq:coordinate-spacetime-integration-by-parts}
    \E\int_0^T\Gamma_Y(f,w_t)(Y_t)\md t =-\int_0^T\int_{\mH}S_u(x)w_t(\mJ x)\md\mu(x)\md t =\int_{[0,T]\times\mR^{\mN}}w\md\sigma_f,
\end{equation}
so it follows that $ \|\sigma_f\|_{\mathrm{TV}}=T\|S_u\|_{L^1(\mu)}$.

Applying \citet[Theorem~3.6(b)]{CCGL2023} with the multiplicatively stable determining class $\mathcal A_0$, and then rescaling, yields an integrable backward extended generator $\overleftarrow{\mathcal L}_{Y,t}$ satisfying
\begin{equation}\label{eq:coordinate-backward-duality}
    \int_{\mR^{\mN}} \bigl(\mathcal L_Yf+\overleftarrow{\mathcal L}_{Y,t}f\bigr)g\md\nu =-\int_{\mR^{\mN}}\Gamma_Y(f,g)\md\nu,\qquad\forall f\in\mathcal A
\end{equation}
at almost every time $t$ for every $g\in\mathcal A$.
By choosing a countable measure-determining family of $\mathcal A$, we may assume that the exceptional time set is independent of $g\in\mathcal A$.
Comparing with \cref{eq:coordinate-spatial-integration-by-parts} identifies the two integrable densities and gives
\begin{equation}\label{eq:coordinate-backward-generator}
    \overleftarrow{\mathcal L}_{Y,t}f(\mathcal Jx) =S_u(x)-Lu(x) =\frac12\Tr(aD^2u(x))+\ip{\bar F(x)}{Du(x)}_{\mH} ={\bar L} u(x)
\end{equation}
for $\md t\otimes\mu$-almost every $(t,x)$.
Since $S_u,Lu\in L^1(\mu)$ for every $u\in\FC$, also ${\bar L} u\in L^1(\mu)$.
It follows that, for every $u=f\circ\mJ$ with $f\in\mathcal A$,
\begin{equation}\label{eq:coordinate-backward-martingale}
    \bar M_t^u\coloneqq u(\bar X_t)-u(\bar X_0)-\int_0^t{\bar L} u(\bar X_s)\md s
\end{equation}
is a martingale according to the boundedness of $u$ and integrability of ${\bar L} u$ on each finite horizon.

For each fixed index $k$, applying \cref{eq:coordinate-backward-martingale} to bounded smooth truncations of the coordinates and their products, followed by localization and the semimartingale product formula, gives \cref{eq:constant-reversed-coordinates,eq:constant-coordinate-covariations} with continuous local martingales $M^k$.
More specifically, we define the coordinate function $\xi^k\coloneqq\ip{\cdot}{e_k}_\mH$.
For every $R>0$, choose $\chi_R\in C_b^\infty(\mR)$ such that $\chi_R(r)=r$ for $|r|\le R$.
Then $\xi_R^k\coloneqq\chi_R\circ\xi^k\in\FC$ is of the form $f\circ\mJ$ with $f\in\mathcal A$, so \cref{eq:coordinate-backward-martingale} applies to $\xi_R^k$.
Consider the stopping time
\begin{equation}
    \tau_R^k\coloneqq\inf\left\{t\ge0\mid|\bar X_t^k|\ge R\right\},\qquad \bar X_t^k\coloneqq\xi^k(\bar X_t).
\end{equation}
Thus for time $s\in\bigl[0,\tau_R^k\bigr]$ we have
\begin{equation}
    {\bar L} \xi_R^k(\bar X_s)=\frac12\Tr\left(aD^2\xi_R^k(\bar X_s)\right)+\ip{\bar F(\bar X_s)}{D\xi_R^k(\bar X_s)}_{\mH}=\ip{\bar F(\bar X_s)}{e_k}_{\mH}
    =\xi^k\circ\bar F(\bar X_s).
\end{equation}
Since $\bar M_t^{\xi_R^k}$ is a martingale, when $\tau_R^k>0$,
\begin{equation}
    \bar M_{t\wedge\tau_R^k}^{\xi_R^k}=\bar X_{t\wedge\tau_R^k}^k-\bar X_0^k-\int_{0}^{t\wedge\tau_R^k}\xi^k\circ\bar F(\bar X_s)\md s
\end{equation}
is also a martingale.
As $R\to+\infty$, we have $\tau_R^k\to+\infty$, and it follows that
\begin{equation}\label{eq:reversed-coordinate-local-martingale}
    M_t^k\coloneqq\bar X_t^k-\bar X_0^k-\int_{0}^{t}\bar F_s^k\md s
\end{equation}
is a continuous local martingale, where we abbreviate $\bar F_s^k\coloneqq\xi^k\circ\bar F(\bar X_s)$.

Analogously, for any pair of indices $(i,j)$, inserting $u=\xi_R^i\xi_R^j$ into \cref{eq:coordinate-backward-martingale} produces another martingale
\begin{equation}\label{eq:stopped-coordinate-product-martingale}
    A_t^{i,j}\coloneqq\bar M_{t\wedge\tau_R^{ij}}^{\xi_R^i\xi_R^j}
    =\bar X_{t\wedge\tau_R^{ij}}^i\bar X_{t\wedge\tau_R^{ij}}^j-\bar X_0^i\bar X_0^j-\int_{0}^{t\wedge\tau_R^{ij}}\left[\bar X_s^j\bar F_s^i+\bar X_s^i\bar F_s^j+q_i\delta_{ij}\right]\md s,
\end{equation}
where the stopping time $\tau_R^{ij}$ is defined as the infimum time whenever $|\bar X_t^i|^2+|\bar X_t^j|^2\ge R^2$.
Meanwhile, we have
\begin{equation}
    \md\bar X_t^k=\bar F_t^k\md t+\md M_t^k
\end{equation}
according to \cref{eq:reversed-coordinate-local-martingale}, so by product formula
\begin{equation}
    \bar X_t^i\bar X_t^j-\bar X_0^i\bar X_0^j=\int_0^t\left(\bar X_s^i\bar F_s^j+\bar X_s^j\bar F_s^i\right)\md s+\int_0^t\bar X_s^i\md M_s^j+\int_0^t\bar X_s^j\md M_s^i+\left[M^i,M^j\right]_t,
\end{equation}
and thus we have a local martingale
\begin{equation}\label{eq:coordinate-product-formula-martingale}
    B_t^{i,j}\coloneqq\bar X_{t\wedge\tau_R^{ij}}^i\bar X_{t\wedge\tau_R^{ij}}^j-\bar X_0^i\bar X_0^j-\int_0^{t\wedge\tau_R^{ij}}\left(\bar X_s^i\bar F_s^j+\bar X_s^j\bar F_s^i\right)\md s-\left[M^i,M^j\right]_{t\wedge\tau_R^{ij}}.
\end{equation}
Taking the difference between \cref{eq:stopped-coordinate-product-martingale} and \cref{eq:coordinate-product-formula-martingale} gives rise to a local martingale
\begin{equation}
    \left[M^i,M^j\right]_{t\wedge\tau_R^{ij}}-\delta_{ij}q_i\left(t\wedge\tau_R^{ij}\right),
\end{equation}
which is clearly a finite-variation process and thus constant.
For any fixed time $t$, we may take $R\to+\infty$, and then \cref{eq:constant-coordinate-covariations} follows.
Furthermore, by setting $i=j=k$, we have $[M^k]_t=q_kt$, which implies that $M^k$ is square integrable.

To reconstruct the noise, set $N_t^{(n)}:=\sum_{k=0}^nM_t^ke_k$.
Doob's inequality gives, for $n>m$,
\begin{equation}\label{eq:coordinate-noise-series-convergence}
    \E\sup_{0\leq t\leq T} \norm{N_t^{(n)}-N_t^{(m)}}_{\mH}^2 \leq4T\sum_{k=m+1}^nq_k\rightarrow0\quad\text{as}\quad m,n\to\infty.
\end{equation}
The partial sums therefore converge in $L^2(\Omega;C([0,T];\mH))$ to a continuous $\mH$-valued martingale $N$ with coordinates $M^k$ and operator quadratic variation $ta$.
Since
\begin{equation}
    BB^*=a, \qquad \int_0^T\norm{B}_{\mathcal L_2(\mU,\mH)}^2\md s =T\Tr a<+\infty,
\end{equation}
the Brownian representation theorem of \citet[Theorem~2 and Corollary~6]{Ondrejat2005Brownian} yields, after adjoining an independent sequence of auxiliary Brownian motions if necessary, a cylindrical Wiener process $\widetilde W$ on $\mU$ such that
\begin{equation}\label{eq:constant-noise-brownian-representation}
    N_t=\int_0^tB\md\widetilde W_s.
\end{equation}

It remains to characterize the drift integrability.
For any nonnegative Borel function of the state, its time integral over $[0,T]$ is invariant under path reversal.
Hence the forward drift integrability in \cref{asm:stationary-diffusion}, together with $\bar F=d-F$ and $2\Firr=2F-d$, shows that \cref{eq:reversed-drift-path-integrability} is equivalent to
\begin{equation}
    \int_0^T\norm{\bar F(\bar X_s)}_{\mH}\md s<+\infty
\end{equation}
almost surely.
Under this condition, the coordinate identities and completeness of $(e_k)$ yield the $\mH$-valued identity
\begin{equation}\label{eq:constant-reversed-integral-equation}
    \bar X_t=\bar X_0+\int_0^t\bar F(\bar X_s)\md s+N_t.
\end{equation}
Together with \cref{eq:constant-noise-brownian-representation}, this proves \cref{eq:constant-reversed-diffusion}.
Conversely, a weak realization of the stationary reversal with the prescribed Bochner-integrable drift requires the last integrability condition, and therefore \cref{eq:reversed-drift-path-integrability}.
The construction is consistent on finite horizons, so the condition for every $T>0$ implies \cref{asm:reversed-diffusion}.

Finally, under the finite-energy condition \cref{eq:finite-cameron-martin-energy}, we have
\begin{equation}\label{eq:energy-drift-integrability}
    \int_\mH\norm{\Firr}_\mH\md\mu \le \left( \int_\mH\norm{B}_{\HS}^2\md\mu \right)^{1/2} \left(\int_\mH g^2\md\mu\right)^{1/2}.
\end{equation}
Stationarity therefore yields \cref{eq:reversed-drift-path-integrability} on every finite time interval.
\end{proof}

\subsection[State-dependent diffusion coefficient]{State-dependent diffusion coefficient}\label{sec:state-dependent-time-reversal}
In this subsection, we identify the stationary reversal for continuous state-dependent diffusion coefficients in the finite-energy regime.
We extend the invariant-measure integration-by-parts identity from cylinder functions to smooth bounded functions on $\mH$.
This allows us to apply the abstract time-reversal criterion directly on $\mH$ and recover the reversed stochastic equation under the norm-integrability assumptions below.

\begin{theorem}\label{thm:continuous-diffusion-time-reversal}
Assume that \cref{asm:stationary-diffusion,asm:cylinder-generator-integrability,asm:covariance-divergence} hold, $B\in L^2(\mu;\mathcal L_2(\mU,\mH))$, and one of $F$ and $\bar F$ belongs to $ L^1(\mu;\mH)$.
Suppose that $B\colon\mH\to\mathcal L_2(\mU,\mH)$ is continuous and that \cref{eq:finite-cameron-martin-energy} holds.
Then the stationary time reversal is an $\mH$-valued weak solution of
\begin{equation}\label{eq:state-dependent-reversed-diffusion}
    \md\bar X_t =\bar F(\bar X_t)\md t +B(\bar X_t)\md\widetilde W_t, \qquad \bar F=d-F=F-2\Firr, \qquad \bar X_0\sim\mu,
\end{equation}
with a Bochner-integrable drift on every finite time interval, on an extension of the stochastic basis if necessary.
In particular, \cref{asm:reversed-diffusion} holds.
\end{theorem}

\begin{proof}
Following the idea of \cref{thm:constant-diffusion-time-reversal}, we apply \citet[Theorem~3.6(b)]{CCGL2023} with configuration space $\mH$ and an appropriate determining class $\mathcal C_0$ as defined below.

When the regular regime \cref{eq:finite-cameron-martin-energy} holds, we have \cref{eq:energy-drift-integrability}.
It follows that both $F$ and $\bar F$ belong to $L^1(\mu;\mH)$, and thus we have $d=2F-2\Firr\in L^1(\mu;\mH)$ as well.

Fix $T>0$.
As in the proof of \cref{thm:constant-diffusion-time-reversal}, let $\bar X$ denote the canonical process with the consistent stationary reversed path laws, equipped with the usual augmentation of its natural filtration.
We extend the differential operators $L(\cdot)$, $S_\cdot$, and $\Gamma(\cdot,\cdot)$ to the function class $\mathcal C\coloneqq C_b^\infty(\mH)$ consisting of all bounded smooth functions with bounded Fr\'echet derivatives of every order, which determines weak convergence of probability measures on $\mH$.
The integrability of $d$ and $F$ indicates
\begin{equation}
    L\varphi,\; S_\varphi,\; \Gamma(\varphi,\psi)\in L^1(\mu),\qquad\forall\varphi,\psi\in\mathcal C.
\end{equation}
By It\^o's formula, $L$ is a forward extended generator on $\mathcal C$ with carr\'e du champ $\Gamma$, whose stochastic-integral term is a square-integrable martingale since $B\in L^2(\mu;\HS(\mU,\mH))$.

We first extend the spatial integration-by-parts identity to $\mathcal C$.
Choose any complete orthonormal basis $(e_k)_{k\ge0}$ of $\mH$, and let $\Pi_n$ be the orthogonal projection onto the spanned subspace of $\{e_0,\ldots,e_{n-1}\}$.
By setting $\varphi_n=\varphi\circ\Pi_n$ and $\psi_n=\psi\circ\Pi_n$, we have $\varphi_n,\psi_n\in\FC$ for $\varphi,\psi\in\mathcal C$.
As in the proof of \cref{thm:constant-diffusion-time-reversal}, the integration-by-parts identity \cref{eq:coordinate-spatial-integration-by-parts} following from \cref{asm:covariance-divergence} becomes
\begin{equation}
    \int_\mH\Gamma(\varphi_n,\psi_n)\md\mu=-\int_\mH S_{\varphi_n}\psi_n\md\mu,
\end{equation}
where
\begin{equation}
    \Gamma(\varphi_n,\psi_n)=\ip{aD\varphi_n}{D\psi_n}_\mH,\qquad S_{\varphi_n}=\Tr(aD^2\varphi_n)+\ip{d}{D\varphi_n}_\mH.
\end{equation}
Since
\begin{equation}
    \begin{aligned}
        |\Gamma(\varphi_n,\psi_n)| &\le \norm{D\varphi}_\infty \norm{D\psi}_\infty\Tr a,\\
        \left|S_{\varphi_n}\right| &\le \norm{D\varphi}_\infty\norm{d}_\mH +\norm{D^2\varphi}_\infty\Tr a,
    \end{aligned}
\end{equation}
dominated convergence yields
\begin{equation}\label{eq:smooth-function-integration-by-parts}
    \int_\mH\Gamma(\varphi,\psi)\md\mu = -\int_\mH S_\varphi\psi\md\mu, \qquad \forall\varphi,\psi\in\mathcal C.
\end{equation}

As in the proof of \cref{thm:constant-diffusion-time-reversal}, we introduce the normalized class
\begin{equation}
    \mathcal C_0 \coloneqq \left\{ \varphi\in\mathcal C: \norm{\varphi}_\infty+\norm{D\varphi}_\infty\le1 \right\}.
\end{equation}
This class is closed under products and still determines weak convergence.
By the continuity of $B$, for fixed $\varphi\in\mathcal C_0$ and $w\in C_b([0,T]\times\mH)$ with $w_t\in\mathcal C_0$, $\Gamma(\varphi,w_t)$ is continuous with respect to $x\in\mH$, and we have the upper bound
\begin{equation}
    |\Gamma(\varphi,w_t)(x)| \le \norm{D\varphi}_\infty \norm{B(x)}_{\HS}^2.
\end{equation}
Hence the spacetime integral is absolutely integrable.
Define $ \tilde\sigma_\varphi(\md t,\md x) \coloneqq-S_\varphi(x)\md t\mu(\md x)$, then
\begin{equation}
    \E\int_0^T\Gamma(\varphi,w_t)(X_t)\md t=-\int_0^T\int_\mH S_\varphi(x)w_t(x)\md\mu(x)\md t=\int_{[0,T]\times\mH}w\md\tilde\sigma_\varphi,
\end{equation}
and we can show that $\|\tilde\sigma_\varphi\|_{\rm TV}=T\|S_\varphi\|_{L^1(\mu)}$ with the same argument as in \cref{eq:coordinate-spacetime-integration-by-parts}.

All hypotheses of \citet[Theorem~3.6(b)]{CCGL2023} are therefore satisfied with configuration space $\mH$ and determining class $\mathcal C_0$.
After rescaling, its integration-by-parts formula gives
\begin{equation}
    \int_\mH \bigl(L\varphi+\overleftarrow{\mathcal L}_t\varphi\bigr) \psi\md\mu = -\int_\mH\Gamma(\varphi,\psi)\md\mu
\end{equation}
for every $\varphi,\psi\in\mathcal C$ at almost every time.
Comparing with \cref{eq:smooth-function-integration-by-parts} and using the countable measure-determining argument from the proof of \cref{thm:constant-diffusion-time-reversal}, we obtain, for each $\varphi\in\mathcal C$,
\begin{equation}
    \overleftarrow{\mathcal L}_t\varphi = S_\varphi-L\varphi = \ip{\bar F}{D\varphi}_\mH +\frac12\Tr(aD^2\varphi) =\bar L\varphi \qquad\md t\otimes\mu\text{-a.e.}
\end{equation}
Thus
\begin{equation}\label{eq:state-dependent-backward-martingale}
    \bar M_t^\varphi\coloneqq\varphi(\bar X_t)-\varphi(\bar X_0) -\int_0^t\bar L\varphi(\bar X_s)\md s
\end{equation}
is a martingale in the reversed natural filtration by the boundedness of $\varphi$ and integrability of $\bar L\varphi$.

Apply the coordinate truncation and product-formula argument from the proof of \cref{thm:constant-diffusion-time-reversal} to \cref{eq:state-dependent-backward-martingale}.
This yields continuous local martingales
\begin{equation}
    M_t^k\coloneqq \ip{\bar X_t-\bar X_0}{e_k}_\mH -\int_0^t\ip{\bar F(\bar X_s)}{e_k}_\mH\md s, \qquad k\in\mN,
\end{equation}
with covariations
\begin{equation}\label{eq:state-dependent-coordinate-covariations}
    \left[M^i,M^j\right]_t = \int_0^t \ip{a(\bar X_s)e_i}{e_j}_\mH\md s.
\end{equation}
In particular,
\begin{equation}
    \E\left[M^k\right]_T = T\int_\mH\norm{B(x)^*e_k}_\mU^2\md\mu(x) <+\infty,
\end{equation}
so each $M^k$ is also square integrable.

To reconstruct the noise, we still set $N_t^{(n)}\coloneqq\sum_{k=0}^nM_t^ke_k$, and it can be shown as in \cref{thm:constant-diffusion-time-reversal} that the partial sums converge to a continuous $\mH$-valued martingale
\begin{equation}
    N_t=\int_0^tB(\bar X_s)\md\widetilde W_s
\end{equation}
with coordinates $M^k$.
The coordinate identities, Bochner integrability of the drift, and completeness of $(e_k)$ now give
\begin{equation}
    \bar X_t = \bar X_0 +\int_0^t\bar F(\bar X_s)\md s +\int_0^tB(\bar X_s)\md\widetilde W_s.
\end{equation}
All martingale identities and integrability bounds hold on every finite horizon of the canonical stationary reversal, so the Brownian representation can be taken on $[0,+\infty)$.
This proves \cref{asm:reversed-diffusion}.
\end{proof}

\section{Consequences and examples}\label{sec:consequences-and-examples}
In this section, we present two examples illustrating how to verify the hypotheses of the preceding results and compute the corresponding entropy production rates.

Combining \cref{thm:entropy-production-dichotomy} with the time-reversal criteria in \cref{thm:constant-diffusion-time-reversal,thm:continuous-diffusion-time-reversal}, we obtain the following summary.
\begin{theorem}\label{thm:verifiable-entropy-production-criteria}
Assume \cref{asm:stationary-diffusion,asm:cylinder-generator-integrability,asm:covariance-divergence}, and let $g$ be defined by \cref{eq:finite-cameron-martin-energy}.
\begin{itemize}
    \item If \cref{eq:finite-cameron-martin-energy} fails, then $\epr=+\infty$;
    \item If \cref{eq:finite-cameron-martin-energy} holds, one of $F$ and $\bar F$ is Lipschitz continuous on bounded sets, and either
    \begin{itemize}
        \item $B$ is constant;
        \item $B\in L^2(\mu;\HS(\mU,\mH))$ is Lipschitz continuous on bounded sets and one of $F$ and $\bar F$ belongs to $L^1(\mu;\mH)$,
    \end{itemize}
    then \cref{asm:reversed-diffusion} holds and for every $T>0$,
    \begin{equation}
        \mP_{[0,T]}\sim\mP^{\mathrm{rev}}_{[0,T]}, \qquad \epr=2\int_\mH g^2\,\md\mu<+\infty.
    \end{equation}
\end{itemize}
\end{theorem}

\begin{example}\label{ex:nonlinear-rotator}
Take
\begin{equation}
    \mU=\mH=\ell^2(\mN;\mR^2)\coloneqq\left\{(x_k)_{k=0}^\infty\mid\sum_{k=0}^\infty\norm{x_k}^2<+\infty,\quad x_k\in\mR^2\quad\forall k\in\mN\right\}
\end{equation}
and introduce
\begin{equation}
    \begin{aligned}
        C&\colon\mH\to\mH,\qquad Cx\coloneqq(c_kx_k)_{k=0}^\infty\quad\text{for}\quad (c_k)_{k=0}^\infty\in\ell^1\quad\text{and}\quad c_k>0\quad\forall k\in\mN;\\
        R&\colon\mH\to\mH,\qquad Rx\coloneqq(r_kJx_k)_{k=0}^\infty\quad\text{for}\quad(r_k)_{k=0}^\infty\in\ell^\infty,\quad J=
        \begin{pmatrix}
            0&-1\\
            1&0
        \end{pmatrix}.
    \end{aligned}
\end{equation}
Consider
\begin{equation}\label{eq:rotator-forward-diffusion}
    \md X_t=F(X_t)\md t+B(X_t)\md W_t,\quad X_0\sim\mu=\mathcal N(0,C),
\end{equation}
where
\begin{equation}
    F(x)=-x+\xi\left(\norm{x}_\mH^2\right)C^{1/2}Rx\quad\text{and}\quad B(x)=(2C)^{1/2}
\end{equation}
for some function $\xi\in C^1([0,+\infty))$ such that
\begin{equation}
    \sup_{s\ge0}|\xi(s)|<+\infty\quad\text{and}\quad\sup_{s\ge0}s|\xi'(s)|<+\infty.
\end{equation}
\end{example}
\begin{proposition}\label{prop:rotator-gaussian-invariance}
The measure $\mu=\mathcal N(0,C)$ is invariant for the transition semigroup of \cref{eq:rotator-forward-diffusion}.
In particular, its solution with initial law $\mu$ is a stationary, time-homogeneous Markov process with continuous $\mH$-valued paths.
\end{proposition}
\begin{proof}
For any $x,h\in\mH$,
\begin{equation}
    DF(x)h=-h+2\xi'\left(\norm{x}_\mH^2\right) \ip{x}{h}_\mH C^{1/2}Rx+\xi\left(\norm{x}_\mH^2\right)C^{1/2}Rh,
\end{equation}
so we obtain the upper bound
\begin{equation}\label{eq:rotator-drift-lipschitz-bound}
    \sup_{x\in\mH}\norm{DF(x)}\le1+\left(2\sup_{s\ge0}s|\xi'(s)|+\sup_{s\ge0}|\xi(s)|\right)\norm{C^{1/2}R}<+\infty.
\end{equation}
Thus $F$ is globally Lipschitz.
Since $F(0)=0$, it also implies
\begin{equation}\label{eq:rotator-drift-linear-growth}
    \norm{F(x)}_\mH\lesssim \norm{x}_\mH, \qquad x\in\mH,
\end{equation}
where we use the symbol ``$\lesssim$'' to indicate inequality up to a constant multiplier.
Moreover,
\begin{equation}\label{eq:rotator-noise-trace}
    \norm{B}_{\HS}^2=2\Tr C=4\sum_{k=0}^\infty c_k<+\infty.
\end{equation}
Therefore, \cref{eq:rotator-forward-diffusion} admits a unique global strong solution with continuous paths and defines a time-homogeneous Markov process.

For any integer $n\ge1$, let $\mH_n$ be the subspace of $\mH$ consisting of first $n$ two-dimensional blocks and let $C_n$, $R_n$ be the restrictions of $C$ and $R$ on $\mH_n$, respectively.
Then we obtain the corresponding truncated equation
\begin{equation}\label{eq:rotator-truncated-diffusion}
    \md X_t^{(n)}=F\left(X_t^{(n)}\right)\md t+(2C_n)^{1/2}\md W_t,\quad X_0^{(n)}\sim\mu_n\coloneqq\mathcal N(0,C_n).
\end{equation}
Denote the density of $\mu_n$ on $\mH_n$ by $p_n$.
Then $\nabla p_n(x)=-p_n(x)C_n^{-1}x$, which together with $R_n^\trans=-R_n$ leads to
\begin{equation}
    \nabla\cdot\left[p_n(x)\xi\left(\norm{x}_\mH^2\right)C_n^{1/2}R_nx\right] =0.
\end{equation}
The Ornstein--Uhlenbeck part also annihilates $p_n$ under its formal adjoint.
Thus $L_n^*p_n=0$, where $L_n$ is the generator of \cref{eq:rotator-truncated-diffusion}.
By uniqueness for the finite-dimensional Fokker--Planck equation with globally Lipschitz drift and constant nondegenerate diffusion, $\mu_n$ is invariant.

Now couple the solutions using the same Wiener process, with $X_0\sim\mu$ and $X_0^{(n)}$ as its projection onto $\mH_n$.
Since $F$ is globally Lipschitz, standard stability estimates for stochastic differential equations, together with the Burkholder--Davis--Gundy inequality, yield
\begin{equation}
    \E\sup_{0\le t\le T}\norm{X_t-X_t^{(n)}}_\mH^2\lesssim\Tr(C-C_n)\to0,\qquad n\to\infty
\end{equation}
for every $T>0$.
Since $\mu_n$ is invariant for \cref{eq:rotator-truncated-diffusion}, $X_t^{(n)}\sim\mu_n$ for every $t\ge0$.
Moreover, $\mu_n$ is the pushforward of $\mu$ under the orthogonal projection onto the first $n$ blocks and hence converges weakly to $\mu$.
Therefore, for every $f\in C_b(\mH)$ and $t\ge0$,
\begin{equation}
    \E f(X_t)=\lim_{n\to\infty}\E f\left(X_t^{(n)}\right) =\lim_{n\to\infty}\int_{\mH_n}f(x)\mu_n(\md x) =\int_\mH f(x)\mu(\md x).
\end{equation}
Therefore $\mu$ is invariant, and stationarity follows from the time-homogeneous Markov property.
\end{proof}
Now we show that \cref{ex:nonlinear-rotator} meets the criteria required in \cref{thm:verifiable-entropy-production-criteria}.

We first verify the standing assumptions.
By \cref{prop:rotator-gaussian-invariance}, the equation has a global strong solution that is a $\mu$-stationary, time-homogeneous Markov process with continuous paths.
Moreover, the Gaussian second moment is
\begin{equation}\label{eq:rotator-gaussian-second-moment}
    \int_\mH\norm{x}_\mH^2\mu(\md x)=\Tr C=2\sum_{k=0}^\infty c_k<+\infty.
\end{equation}
So by \cref{eq:rotator-drift-linear-growth},
\begin{equation}\label{eq:rotator-drift-square-integrability}
    \int_\mH\norm{F(x)}_\mH^2\mu(\md x)\lesssim\int_\mH\norm{x}_\mH^2\mu(\md x)=\Tr C<+\infty,
\end{equation}
which implies that $F\in L^2(\mu;\mH)\subseteq L^1(\mu;\mH)$.
Hence stationarity and \cref{eq:rotator-noise-trace} yield
\begin{equation}\label{eq:rotator-path-integrability}
    \mathbb E\int_0^T\left(\norm{F(X_s)}_\mH+\norm{B(X_s)}_{\HS}^2\right)\md s=T\left[\norm{F}_{L^1(\mu;\mH)}+\Tr(2C)\right]<+\infty
\end{equation}
for every $T>0$.
In particular, the required time integral is finite almost surely, so \cref{asm:stationary-diffusion} holds.

Next, for any test function $\varphi\in\FC$,
\begin{equation}
    \norm{L\varphi}_{L^2(\mu)} \le\norm{D\varphi}_\infty\norm{F}_{L^2(\mu;\mH)} +\frac12\norm{D^2\varphi}_\infty\norm{B}_{\HS}^2,
\end{equation}
which is also finite according to \cref{eq:rotator-drift-square-integrability,eq:rotator-noise-trace}.
Consequently, we have $L\varphi\in L^2(\mu)\subseteq L^1(\mu)$, which implies \cref{asm:cylinder-generator-integrability}.

To identify the drift correction $d$ in \cref{asm:covariance-divergence}, Gaussian integration by parts gives
\begin{equation}
    \int_\mH\ip{ah}{D\varphi}_\mH\md\mu=2\int_\mH\ip{Ch}{D\varphi}_\mH\md\mu=2\int_\mH\ip{x}{h}_\mH\varphi\md\mu,\quad\forall h\in\mH,\,\varphi\in\FC,
\end{equation}
which means $d=-2x$.
By \cref{eq:rotator-gaussian-second-moment}, $d\in L^2(\mu;\mH)\subseteq L^1(\mu;\mH)$, leading to \cref{asm:covariance-divergence}.
The associated drift decomposition is therefore
\begin{equation}\label{eq:rotator-drift-decomposition}
    F(x)=\underbrace{(-x)}_{\Frev(x)}+\underbrace{\xi\left(\norm{x}_\mH^2\right)C^{1/2}Rx}_{\Firr(x)} \qquad \bar F(x)=d(x)-F(x)
    =\underbrace{(-x)}_{\Frev(x)}-\underbrace{\xi\left(\norm{x}_\mH^2\right)C^{1/2}Rx}_{\Firr(x)}.
\end{equation}

We can now check the time-reversal criteria in \cref{thm:verifiable-entropy-production-criteria}.
Since $B=(2C)^{1/2}$ is constant and $F$ is globally Lipschitz, it suffices to evaluate the extended Cameron--Martin energy directly.
By \cref{eq:rotator-drift-decomposition}, we obtain $\Firr\in\Ran(a^{1/2})$, and then direct calculation gives
\begin{equation}\label{eq:rotator-cameron-martin-energy}
    g(x)^2=\norm{a^{-1/2}F_{\mathrm{irr}}(x)}_\mH^2=\frac12\xi\left(\norm{x}_\mH^2\right)^2\sum_{k=0}^\infty r_k^2\norm{x_k}^2.
\end{equation}
Using $\int_\mH\norm{x_k}_{\mR^2}^2\,\md\mu=2c_k$ for each $k$ and Tonelli's theorem, we obtain
\begin{equation}
    \int_\mH g^2\md\mu\le\sup_{s\ge0}|\xi(s)|^2\sum_{k=0}^\infty r_k^2c_k \lesssim\norm{(r_k)_{k=0}^\infty}_{\ell^\infty}^2\norm{(c_k)_{k=0}^\infty}_{\ell^1}<+\infty.
\end{equation}
Thus the finite-energy condition holds, and we have
\begin{equation}
    \Prob_{[0,T]}\sim\Prob^{\mathrm{rev}}_{[0,T]}\quad\text{and}\quad\epr=2\int_\mH g^2\md\mu
    =\int_\mH\xi\left(\norm{x}_\mH^2\right)^2\sum_{k=0}^\infty r_k^2\norm{x_k}^2\mu(\md x)<+\infty
\end{equation}
by \cref{thm:verifiable-entropy-production-criteria}.
\begin{remark}
This example actually provides a concrete case when $\Firr\in\Ran(a^{1/2})\setminus\Ran(a)$ with a finite positive EPR value.
Let $r_k=1$ for any $k$ and $\xi(s)=1/(1+s)$, then we have
\begin{equation}
    \epr=2\int_\mH g^2\md\mu=\int_\mH\xi\left(\norm{x}_\mH^2\right)^2\sum_{k=0}^\infty r_k^2\norm{x_k}^2\mu(\md x)
    =\E_\mu\left[\frac{\norm{x}_\mH^2}{\left(1+\norm{x}_\mH^2\right)^2}\right]\in\left(0,\frac14\right].
\end{equation}
Meanwhile, since $a=2C$ and $\xi>0$,
\begin{equation}
    \Firr(x)\in\Ran(a)\iff\sum_kc_k^{-1}\norm{x_k}^2<+\infty.
\end{equation}
However, for $\mu=\mathcal N(0,C)$, the summands $c_k^{-1}\norm{x_k}^2$ are independently and identically distributed with a $\chi_2^2$ distribution, whose summation is infinite almost surely.
Therefore, we have $\Firr\notin\Ran(a)$ $\mu$-almost every $x$.
\end{remark}

\begin{example}\label{ex:state-dependent-blocks}
Take $\mU=\mH=\ell^2(\mN;\mR^2)$, a linear operator $C$, and a $2\times2$ matrix $J$ as defined in \cref{ex:nonlinear-rotator}.
Consider
\begin{equation}\label{eq:block-forward-diffusion}
    \md X_t=F(X_t)\md t+B(X_t)\md W_t,\quad X_0\sim\mu=\mathcal N(0,C),
\end{equation}
where
\begin{equation}
    F(x)=\bigl(-(1+\theta)m_k(x_k)x_k+(1-\theta) Jx_k\bigr)_{k=0}^\infty\quad\text{and}\quad B(x)u=\bigl(\sqrt{2c_km_k(x_k)}u_k\bigr)_{k=0}^\infty
\end{equation}
for any $x\in\mH$ and $u\in\mU$ with $m_k(x_k)=\exp\left(-\theta\|x_k\|^2/(2c_k)\right)$ and a parameter $\theta\in[0,1]$.
\end{example}
\begin{proposition}\label{prop:block-gaussian-invariance}
The measure $\mu=\mathcal N(0,C)$ is invariant for the transition semigroup of \cref{eq:block-forward-diffusion}.
In particular, its solution with initial law $\mu$ is a stationary, time-homogeneous Markov process with continuous $\mH$-valued paths.
\end{proposition}
\begin{proof}
For each coordinate $F_k$ of $F$, we have
\begin{equation}
    \begin{aligned}
        \norm{DF_k(x_k)}&\le(1+\theta)\me^{-\theta\norm{x_k}^2/(2c_k)}\left\|I_2-\frac{\theta}{c_k}x_kx_k^\trans\right\|+(1-\theta)\|J\|\\
        &=(1+\theta)\me^{-\theta\norm{x_k}^2/(2c_k)}\left(1\vee\left|1-\theta\|x_k\|^2/c_k\right|\right)+(1-\theta)\le2.
    \end{aligned}
\end{equation}
So for all $x,y\in\mH$,
\begin{equation}
    \norm{F(x)-F(y)}_\mH^2 =\sum_{k=0}^\infty\norm{F_k(x_k)-F_k(y_k)}^2 \le4\sum_{k=0}^\infty\norm{x_k-y_k}^2 =4\norm{x-y}_\mH^2,
\end{equation}
by the mean value estimate in each block, and thus $F$ is globally Lipschitz.
Since $F(0)=0$, it also implies
\begin{equation}\label{eq:block-drift-linear-growth}
    \norm{F(x)}_\mH\lesssim \norm{x}_\mH, \qquad x\in\mH.
\end{equation}
Writing the \(k\)-th block of \(B\) as \(b_k(z)I_2\), where $b_k(z)\coloneqq\sqrt{2c_km_k(z)}$, we have
\begin{equation}
    \norm{\nabla b_k(z)}^2=\frac{\theta^2}{2c_k}\norm{z}^2\me^{-\frac{\theta}{2c_k}\norm{z}^2}\le\theta\me^{-1},
\end{equation}
so
\begin{equation}
    \norm{B(x)-B(y)}_{\HS}^2=2\sum_{k=0}^\infty\left|b_k(x_k)-b_k(y_k)\right|^2\le2\theta\me^{-1}\norm{x-y}_\mH^2.
\end{equation}
Meanwhile, we have
\begin{equation}\label{eq:block-noise-trace-bound}
    \norm{B(x)}_{\HS}^2 =\sum_{k=0}^\infty4c_km_k(x_k) \le4\sum_{k=0}^\infty c_k<+\infty,\quad\forall x\in\mH.
\end{equation}
Therefore, \cref{eq:block-forward-diffusion} admits a unique global strong solution with continuous paths and defines a time-homogeneous Markov process.

For each fixed $k$, consider
\begin{equation}
    \md Y_t=-(1+\theta)m_k(Y_t)Y_t\md t+\sqrt{2c_km_k(Y_t)}\md W_t,\quad Y_t\in\mR^2.
\end{equation}
A direct computation shows that the generator of the radial part can be written as
\begin{equation}
    \mathscr S_Y:\varphi(z)\mapsto-(1+\theta)m_k(z)z\cdot\nabla\varphi(z)+c_km_k(z)\Delta \varphi(z)=p_k^{-1}(z)\nabla\cdot\left(c_km_k(z)p_k(z)\nabla\varphi(z)\right),
\end{equation}
where we set
\begin{equation}
    p_k(z)\coloneqq\frac1{2\pi c_k}\me^{-\frac{\norm{z}^2}{2c_k}},\quad z\in\mR^2.
\end{equation}
Hence, for $f,g\in C_c^\infty(\mathbb R^2)$,
\begin{equation}
    -\int_{\mathbb R^2} f (\mathscr S_Y g)p_k\md z=\int_{\mathbb R^2} c_km_k(\nabla f)\cdot(\nabla g)p_k\md z,
\end{equation}
so the generator is symmetric in $L^2(\mu_k)$, where $\mu_k=\mathcal N(0,c_kI_2)$.
Since the coefficients are globally Lipschitz, standard finite-dimensional diffusion theory, together with the above symmetric divergence form, yields a unique $\mu_k$-reversible diffusion $Y_k$.

Since $\exp((1-\theta)tJ)$ is orthogonal and $m_k$ is radial, It\^o's formula gives
\begin{equation}
    \md X_{k,t}=\left[ -(1+\theta)m_k(X_{k,t})X_{k,t} +(1-\theta)JX_{k,t} \right]\md t+ \sqrt{2c_km_k(X_{k,t})} \exp((1-\theta)tJ)\md W_{k,t}
\end{equation}
with $X_{k,t}\coloneqq \exp((1-\theta)tJ)Y_{k,t}$.
Define
\begin{equation}
    \widetilde W_{k,t} \coloneqq \int_0^t \exp((1-\theta)sJ)\md W_{k,s}.
\end{equation}
Since $\exp((1-\theta)sJ)$ is orthogonal, $\widetilde W_k$ is again a two-dimensional Brownian motion.
Therefore,
\begin{equation}
    \md X_{k,t} = \left[ -(1+\theta)m_k(X_{k,t})X_{k,t} +(1-\theta)JX_{k,t} \right]\md t + \sqrt{2c_km_k(X_{k,t})}\md\widetilde W_{k,t},
\end{equation}
which corresponds to the dynamics for the $k$-th block in \cref{eq:block-forward-diffusion}.
Moreover, rotational invariance of $\mu_k$ implies that $X_{k,t}\sim\mu_k$ for all $t\ge0$.

Finally, recall that the coefficients in \cref{eq:block-forward-diffusion} are block diagonal and that $\mu=\bigotimes_{k=0}^\infty\mu_k$.
Since the initial blocks are independent under $\mu$ and the coordinate Brownian motions are independent, the block processes remain independent.
Together with $X_{k,t}\sim\mu_k$ for any $t\ge0$ and every $k$, this yields $X_t\sim\mu$ for all $t\ge0$.
Thus $\mu$ is invariant for the Hilbert-space diffusion.
Since the latter is time-homogeneous Markov, the solution started from $X_0\sim\mu$ is $\mu$-stationary.
\end{proof}
Now we show that \cref{ex:state-dependent-blocks} meets the criteria required in \cref{thm:verifiable-entropy-production-criteria}.

With the same deduction as in the previous example, \cref{asm:stationary-diffusion,asm:cylinder-generator-integrability} can be directly verified by \cref{prop:block-gaussian-invariance,eq:block-drift-linear-growth,eq:block-noise-trace-bound}, and it remains to identify the drift correction $d$ in \cref{asm:covariance-divergence}.
Applying the preceding finite-dimensional integration-by-parts identity blockwise results in
\begin{equation}
    \int_\mH \ip{a(x)h}{D\varphi(x)}_\mH \mu(\md x) = - \int_\mH \ip{d(x)}{h}_\mH \varphi(x)\mu(\md x), \qquad \forall\varphi\in\FC,\, h\in\mH
\end{equation}
where the $k$-th block of $d$ writes
\begin{equation}
    d_k(z)=p_k(z)^{-1}\nabla\cdot(a_k(z)p_k(z))=2c_kp_k(z)^{-1}\nabla(m_k(z)p_k(z))=-2(1+\theta)m_k(z)z
\end{equation}
and \(a_k(z)=2c_km_k(z)I_2\).
It follows that we have the upper bound
\begin{equation}
    \int_\mH\norm{d(x)}_\mH^2\mu(\md x)\le 4(1+\theta)^2\int_\mH\norm{x}_\mH^2\mu(\md x)=4(1+\theta)^2\Tr C<+\infty,
\end{equation}
so $d\in L^2(\mu;\mH)\subseteq L^1(\mu;\mH)$ and \cref{asm:covariance-divergence} is verified.
The associated drift decomposition is therefore
\begin{equation}\label{eq:block-drift-decomposition}
    F_k(x_k)=\underbrace{-(1+\theta)m_k(x_k)x_k}_{F_{\rev,k}(x_k)}+\underbrace{(1-\theta)Jx_k}_{F_{\irr,k}(x_k)} \qquad \bar F_k(x_k)
    =\underbrace{-(1+\theta)m_k(x_k)x_k}_{F_{\rev,k}(x_k)}-\underbrace{(1-\theta)Jx_k}_{F_{\irr,k}(x_k)}.
\end{equation}

The estimates above also give $F,\bar F\in L^1(\mu;\mH)$ and $B\in L^2(\mu;\mathcal L_2(\mU,\mH))$.
We now evaluate the extended Cameron--Martin energy and apply the entropy-production lower bound.
For each $x\in\mH$, $\Firr\in\Ran(a^{1/2})$ if and only if
\begin{equation}\label{eq:block-cameron-martin-energy}
    g(x)^2=\norm{a^{-1/2}F_{\mathrm{irr}}(x)}_\mH^2=\sum_{k=0}^\infty\frac{(1-\theta)^2}{2c_km_k(x_k)}\norm{x_k}^2
    =\sum_{k=0}^\infty(1-\theta)^2 Z_k(x)\me^{\theta Z_k(x)}<+\infty,
\end{equation}
where we introduce $Z_k(x)\coloneqq\|x_k\|^2/(2c_k)$.
Under distribution $\mu$, the variables $Z_k$ are independent and identically distributed with an $\Exp(1)$ distribution, so the strong law of large numbers implies that $N^{-1}\sum_{k=0}^{N-1}Z_k\to1$ $\mu$-almost surely as $N\to+\infty$.
Therefore, $\Firr\in\Ran(a^{1/2})$ for $\mu$-a.e. $x$ if and only if $\theta=1$.

Consequently, for the general case when $\theta\in[0,1)$, $\Firr\notin\Ran(a^{1/2})$ for $\mu$-almost every $x$, and \cref{thm:verifiable-entropy-production-criteria} gives $\epr=+\infty$.
Otherwise when $\theta=1$, we have $\epr=2\int g^2\md\mu=0$.
\begin{remark}
There is also an interesting pattern behind this example.
For $0\leq\theta<1$, consider the projection onto the first $N$ blocks and denote its Cameron--Martin energy by $(g^{[N]})^2$.
Then
\begin{equation}
    \int_\mH g^{[N]}(x)^2\md\mu=(1-\theta)^2\sum_{k=0}^{N-1} \int_\mH Z_k(x)\me^{\theta Z_k(x)}\md\mu=(1-\theta)^2\sum_{k=0}^{N-1}\E\left[Z\me^{\theta Z}\right]=N,\quad Z
    \sim\Exp(1),
\end{equation}
and it follows that $\epr^{[N]}=2N$.
Therefore, this example has also offered an instance where every finite truncation admits a finite EPR value, but the total EPR is infinite.
\end{remark}

\section{Conclusion}\label{sec:conclusion}

This work extends the path-space theory of entropy production from finite-dimensional stationary diffusions to diffusions on separable Hilbert spaces.
The passage to infinite dimensions is not merely formal.
To address the resulting difficulties, we formulate the reversible--irreversible decomposition directly with respect to the invariant measure, separate the identification of the reversed dynamics from the entropy calculation, and measure irreversibility through the Cameron--Martin geometry induced by the noise.
This yields an extended-valued lower bound that detects infinite entropy production, together with sufficient conditions for the matching identity in the finite-energy regime.

The resulting picture suggests that, for stochastic evolution equations, irreversibility should be understood as a joint property of the dynamics and the geometry of the noise, rather than of the drift alone.
The examples highlight two consequences that are particularly relevant in infinite dimensions: irreversible directions outside the covariance range may still have finite entropic cost, while finite entropy production at every finite-dimensional resolution does not by itself control the entropy production of the full path law.
These phenomena indicate that the noise-induced geometry provides a natural framework for understanding time-irreversibility in stochastic evolution equations.

More broadly, this perspective may help guide further work on entropy production for SPDEs and other infinite-dimensional Markov systems.
Rather than relying on model-specific finite-dimensional analogies, it suggests seeking structural criteria that remain meaningful under approximation, projection, and increasingly singular dynamics.
In particular, it would be natural to understand when finite-dimensional entropy-production calculations faithfully approximate the full system, and how irreversible behavior is preserved or altered in infinite-dimensional limits.
Such questions may help connect probabilistic time-reversal theory, stochastic thermodynamics, and the analysis of stochastic evolution equations within a common framework.

\section*{Acknowledgements}
This project is supported by the National Research Foundation, Singapore, under its AI Singapore Programme (AISG Award No.: AISG3-RP-2022-028).
Z.L. is supported by the Ministry of Education, Singapore, under its funding for the Research Centre of Excellence award to the Institute for Functional Intelligent Materials (Project No. EDUNC-33-18-279-V12-IFIM).
Y.Z. is supported by the National University of Singapore and the AI for Science Institute, Beijing, through the AISI--NUS Joint Research Initiative Fund 2025 Award.
A.Z. is supported by the National Research Foundation, Singapore, under the NRF fellowship (Project No. NRF-NRFF13-2021-0005).

\bibliographystyle{plainnat}
\bibliography{references}
\end{document}